\documentclass[11pt, twoside]{article}
\usepackage[margin=1in]{geometry}

\usepackage{tikz}
\usetikzlibrary{calc}

\usepackage{amsmath,amssymb,amsthm}

\usepackage[utf8]{inputenc}

\usepackage{enumerate}

\usepackage{fancyhdr}
\usepackage{cite}

\usepackage{listings}
\newcommand{\TD}{\mathcal T}

\newtheorem{algorithm}{Algorithm}

\usepackage{hyperref}
\hypersetup{%
    colorlinks=True, 
    citecolor=blue,
    linkcolor=black}

\usepackage[charter]{mathdesign}

\usepackage{accents}

\newcommand{\rom}[1]{\uppercase\expandafter{\romannumeral #1\relax}}
  
  \theoremstyle{definition}
  \newtheorem{theorem}{Theorem}[section]

  \newtheorem{lemma}[theorem]{Lemma}
  \newtheorem{definition}[theorem]{Definition}
  \newtheorem{remark}[theorem]{Remark}

\newtheorem{thmx}{Assumption}
\newtheorem{assumption}[thmx]{Assumption}

  \newtheorem*{assumption*}{Assumption}

  \numberwithin{equation}{section}

\begin{document}

\title{A multi-material topology optimization algorithm based on the topological derivative}

\author{P. Gangl\footnote{Institute of Applied Mathematics, Graz University of Technology, Steyrergasse 30, 8010 Graz, Austria.  E-mail: gangl@math.tugraz.at} }

\date{\today}

\maketitle

\begin{abstract}
	We present a level-set based topology optimization algorithm for design optimization problems involving an arbitrary number of different materials, where the evolution of a design is solely guided by topological derivatives. Our method can be seen as an extension of the algorithm that was introduced in \cite{AmstutzAndrae2006} for two materials to the case of an arbitrary number $M$ of materials. We represent a design that consists of multiple materials by means of a vector-valued level set function which maps into $\mathbb R^{M-1}$. We divide the space $\mathbb R^{M-1}$ into $M$ sectors, each corresponding to one material, and establish conditions for local optimality of a design based on certain generalized topological derivatives. The optimization algorithm consists in a fixed point iteration striving to reach this optimality condition. Like the two-material version of the algorithm, also our method possesses a nucleation mechanism such that it is not necessary to start with a perforated initial design. We show numerical results obtained by applying the algorithm to an academic example as well as to the compliance minimization in linearized elasticity.
\end{abstract}

\section{Introduction}
Over the past decades, numerical shape and topology optimization techniques have become an integral part in the design process not only of mechanical structures, but also in applications from electromagnetics, fluid mechanics and many more. While classical shape optimization approaches \cite{DZ2} can only alter boundaries or material interfaces of a given design, topology optimization approaches are more general and can yield optimal designs of any topology. There are several classes of methods for optimizing the topology of a design. A thorough overview over the different approaches to topology optimization is given in the review article \cite{MauteSigmund2013}.
% In the homogenization method \cite{Allaire2002}, the domain is represented as a periodic microstructure where the strucutre of each cell is determined by a small number of parameters.

The idea of density-based topology optimization \cite{BendsoeSigmund2003} is to represent a design by means of a density variable $\rho$, which can attain any value between $0$ and $1$, where regions with $\rho(x) = 1$ are interpreted as one material and regions with $\rho(x) = 0$ as the other material (e.g. an isotropic material and void). In order to avoid large areas of intermediate density values $0 < \rho(x) < 1$, penalization is performed in the constitutive equation, which -- in combination with a constraint on the allowed volume -- makes ``black'' ($\rho(x) =1$) or ``white'' ($\rho(x) =0$) regions more favorable and therefore removes the ``gray'' ($0 < \rho(x) < 1$) areas. In order to avoid numerical instabilities, which occur as a consequence of the non-existence of solutions to the topology optimization problem, regularization techniques such as density filtering, sensitivity filtering or bounding the perimeter of the structure are used \cite{SigmundPetersson1998}. Density-based topology optimization approaches are among the most widely used topology optimization approaches and have been successfully applied to a large number of practical applications, see e.g. \cite{AageEtAlNature2017}.

In contrast to density-based methods, level set approaches \cite{OsherSethian:1988a} do not introduce intermediate material properties. Here, the domain of interest $\Omega$ is represented by means of a continuous function $\psi$ which attains negative values inside $\Omega$ and positive values outside $\Omega$. Thus, the boundary of $\Omega$ is given as the zero level set of $\psi$, i.e. $\partial \Omega = \{x: \psi(x) = 0 \}$. In the level set method for shape and topology optimization \cite{AllaireJouveToader2004}, the evolution of the level set function $\psi$ is guided by shape gradients by means of a Hamilton-Jacobi equation. The implicit geometry description by means of a level set function brings a lot of flexibility in the treatment of topological changes. However, since no topological sensitivity information is included, the method lacks a nucleation mechanism. Holes or different components can merge, but no holes can be created in the interior of the design. This problem is typically circumvented by choosing a perforated initial design with many circular holes. Also this approach has proven very useful in many practical applications, see e.g. \cite{FepponAllaireDapognyEtAl2019}.

As an alternative to choosing a perforated initial design in the level set method, a coupling of the Hamilton-Jacobi equation, which uses shape sensitivities, with the \textit{topological derivative} was proposed in \cite{AllaireJouve2006, BurgerHacklRing2004}. The topological derivative \cite{NovotnySokolowski2013} at a point $z$ indicates whether a change of material at that point will yield an increase or decrease of the objective function.

As opposed to the level set method proposed in \cite{AllaireJouveToader2004}, the algorithm introduced in \cite{AmstutzAndrae2006, Amstutz2011a}, which also uses a level set representation of the domain, exploits topological sensitivity information. Here, the evolution of the level set function is guided by the topological derivative. We will give a thorough introduction to this method in Section \ref{sec_algo_revisited}. A related topology optimization approach is the one introduced in \cite{YamadaIzuiNishiwakiTakezawa2010}, where instead of the topological derivative the sensitivity of the objective with respect to a variation of the level set function is used.

Most applications of and methods for topology optimization deal with finding the optimal distribution of two different phases (e.g. one isotropic material and void, or two different materials) within a given design area. However, most of the approaches mentioned above have been extended to the multi-material case where one is interested in distributing a certain number of different materials within the design region in an optimal way. Multi-material topology optimization in a density-based setting was investigated in \cite{HvejselLund2011}. The level set method based on shape sensitivities and a Hamilton-Jacobi equation was generalized to the multi-material case in \cite{WangWang2004,AllaireDapognyMichailidis2014, WangCMAME2015} and similar techniques have also been used in image segmentation \cite{VeseChan2002}. Multi-material topology optimization has also been used in a phase field setting \cite{BlankEtAl2014} and from an optimal control point of view \cite{ClasonKunisch2016}. Moreover, the approach introduced in \cite{YamadaIzuiNishiwakiTakezawa2010} has been extended to the multi-material case, e.g., in \cite{LimMin2012}.

In this work, we propose an alternative way of multi-material topology optimization which is based solely on topological derivatives. Our algorithm can be seen as a generalization of the work presented in \cite{AmstutzAndrae2006, Amstutz2011a} to the case of multiple materials. We use a description of the design by a vector-valued level set function. The goal of the method is to reach a local optimality condition which is expressed by means of topological derivatives. By construction, our algorithm is capable of altering shape and topology of an initial design. Holes can be nucleated and new components can be created. Therefore, there is no need to start with perforated initial designs.

The rest of this paper is organized as follows: In Section \ref{sec_algo_revisited} we will revisit the algorithm introduced in \cite{AmstutzAndrae2006} and point out its main ingredients. In Section \ref{sec_algo_m_mat} we will first introduce a level set framework for multiple materials and then generalize the algorithm introduced in \cite{AmstutzAndrae2006} to the case of an arbitrary number of materials. Finally, we present numerical examples in Section \ref{sec_numerics}.

%%%%%%%%%%%%%%%%%%%%%%%%%%%%%%%%%%%%%%%%%%%%%%%%%%%%%%%%%%%%%%%%
% \section{ \cite{AmstutzAndrae2006} revisited} 
\section{Two-material topology optimization using topological derivatives}\label{sec_algo_revisited}
In this section, we revisit the algorithm introduced in \cite{AmstutzAndrae2006} and recall its main ingredients. The algorithm uses a level set description of the design and the evolution of the design is guided solely by the topological derivative, allowing for nucleation of holes or creation of new components anywhere in the design domain.

\subsection{Topological derivative}
The topological derivative of a domain-dependent shape function represents its sensitivity with respect to a topological perturbation of the domain. The idea of the topological derivative was first used in \cite{EschenauerKobelevSchumacher1994} as the ``bubble method'' and was introduced in a mathematically rigorous way for the first time in \cite{SokolowskiZochowski1999}, see also \cite{NovotnySokolowski2013, NovotnySokolowskiZochowski2019} for an overview on the topic.

In the following, let $d \in \{1,2,3\}$ denote the space dimension and let an open hold-all domain $D \subset \mathbb R^d$ be given. Let $\mathcal P(D)$ denote the power set of $D$, i.e., the set of all subsets of $D$, and let $\mathcal A \subset \mathcal P(D)$ denote a set of admissible subsets of $D$, the definition of which may depend on the problem at hand. Let a shape function $\mathcal J$
\begin{align} 
	\begin{aligned}
		\mathcal J: \mathcal A &\rightarrow \mathbb R, \\
		\Omega &\mapsto \mathcal J(\Omega),
	\end{aligned}
\end{align}
be given. Let us further fix an open set $\Omega_1 \in \mathcal A$ and define $\Omega_2 := D \setminus \overline \Omega_1$. We define the topological derivative of the shape function $\mathcal J$ at a point $z\in \Omega_1$ with respect to a change of the material in a neighborhood $\omega_\varepsilon$ of $z$, see Fig. \ref{fig_setting} (left).

\begin{definition} \label{def_TD}
Let $\mathcal J: \mathcal A \rightarrow \mathbb R$ a shape function and $\Omega_1 \in \mathcal A$ an open admissible set. Let further $\omega \subset \mathbb R^d$ open with $0 \in \omega$.
Let $z \in \Omega_1$ and, for $\varepsilon>0$ small, let $\omega_\varepsilon := z + \varepsilon \omega$. Then, the topological derivative of $\mathcal J$ at the point $z \in \Omega_1$ is defined as the limit
\begin{align*}
	\TD^{1\rightarrow 2}(z) := \underset{\varepsilon \searrow 0}{\mbox{lim }} \frac{\mathcal J(\Omega_1 \setminus \overline \omega_\varepsilon ) - \mathcal J(\Omega_1) }{ \ell(\varepsilon) }
\end{align*}
where $\ell : \mathbb R^+ \rightarrow \mathbb R^+$ is a continuous positive function satisfying $\ell(\varepsilon) \rightarrow 0$ as $\varepsilon \rightarrow 0$.
\end{definition}
Here, the set $\omega$ determines the shape of the small inclusion $\omega_\varepsilon$. The most common choice is $\omega = B(0,1)$ to deal with ball-shaped inclusions, but also other shapes such as ellipses are possible. The function $\ell(\varepsilon)$ has to be chosen depending on the concrete application at hand. For most applications one has to choose $ \ell(\varepsilon) = |\omega_\varepsilon| = \varepsilon^d|\omega|$, however certain applications require different choices of $\ell(\varepsilon)$ \cite{NovotnySokolowski2013}.

Similarly to Definition \ref{def_TD}, by interchanging the roles of $\Omega_1$ and $\Omega_2$, the topological derivative can also be defined for a point $\tilde z \in \Omega_2$ with the corresponding inclusion $\tilde \omega_\varepsilon:= \tilde z + \varepsilon \omega$. Then, the topological derivative reads
\begin{align*}
	\TD^{2\rightarrow 1}(\tilde z) := \underset{\varepsilon \searrow 0}{\mbox{lim }} \frac{\mathcal J(\Omega_1 \cup \tilde \omega_\varepsilon ) - \mathcal J(\Omega_1) }{ \ell(\varepsilon) }.
\end{align*}

\begin{figure}
    \begin{tabular}{cc}
        \includegraphics[width=.48\textwidth]{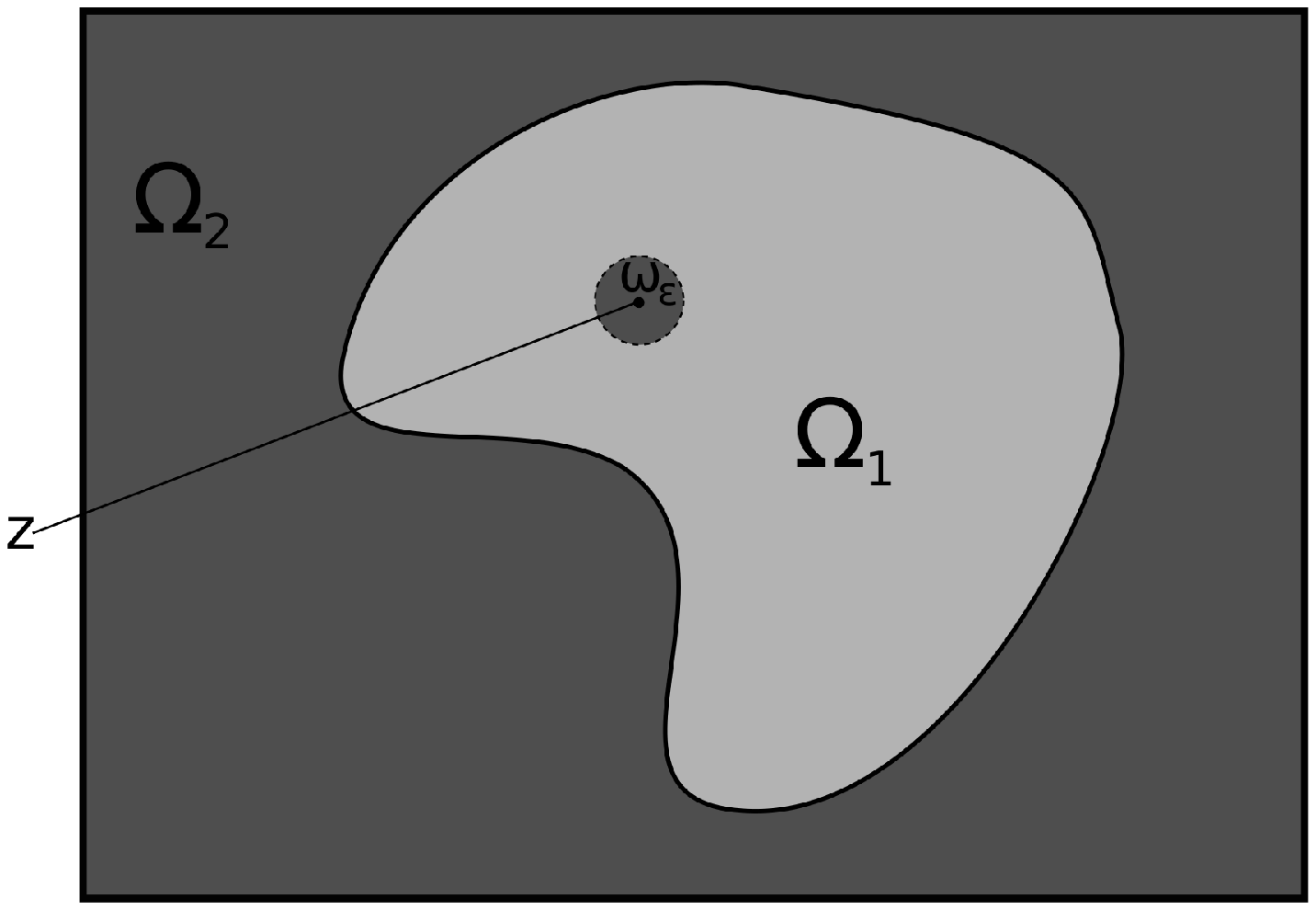} & \includegraphics[width=.5\textwidth]{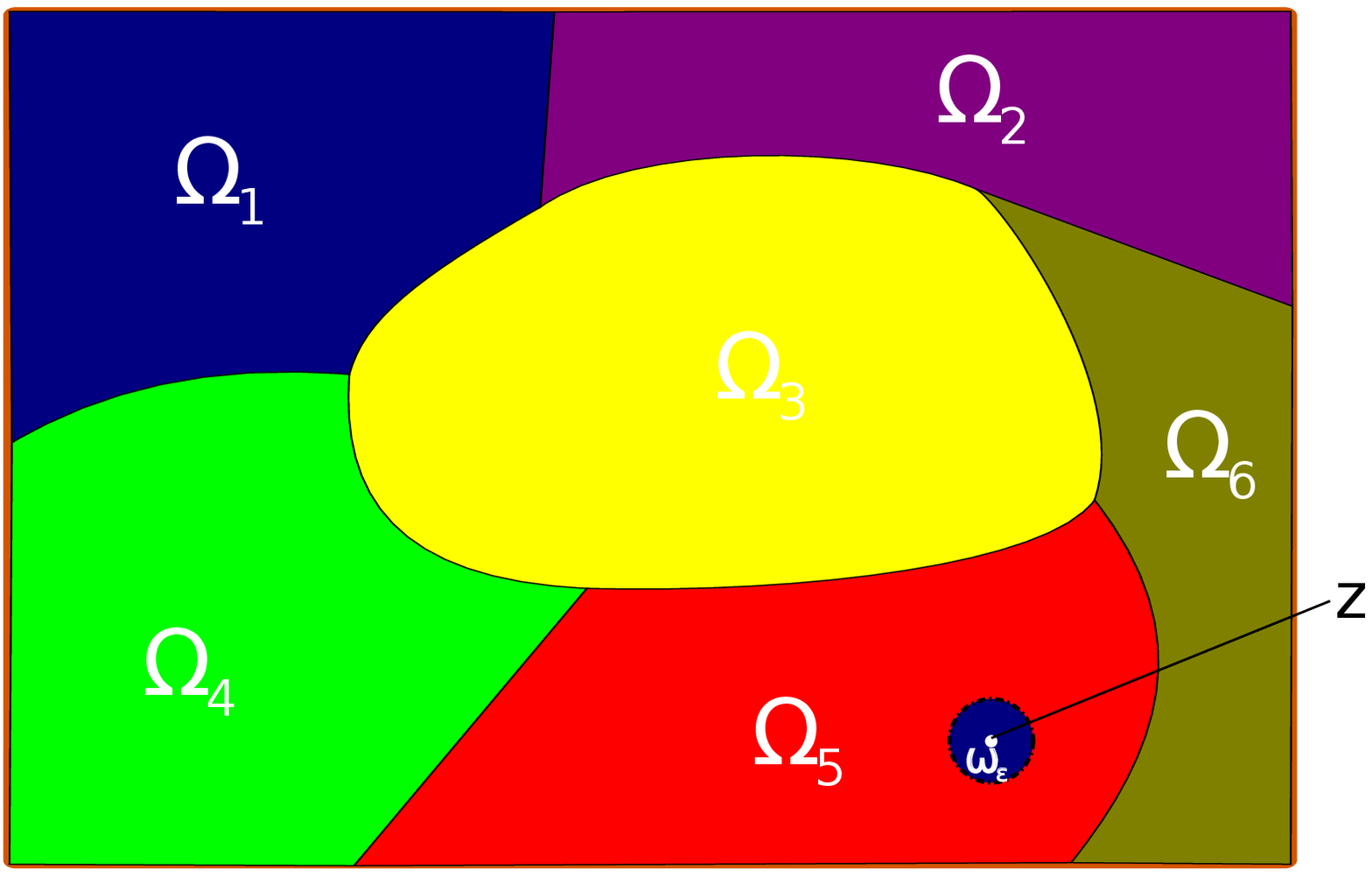}
    \end{tabular}
    \caption{Illustration of topological derivative $\TD^{1 \rightarrow 2}$ in case of two materials (left) and of topological derivative $\TD^{5 \rightarrow 1}$ in case of multiple (here: six) materials (right).}
    \label{fig_setting}
\end{figure}

\begin{remark}
The mathematically rigoros derivation of topological derivatives for shape functions that depend on the shape via the solution to a partial differential equation is a research field on its own. There exist different approaches for the derivation of topological derivatives, see e.g. \cite{NovotnySokolowski2013, NovotnySokolowskiZochowski2019, Amstutz2006, GanglSturm2019a}. Since the current work is concerned with an optimization algorithm based on topological derivatives, we assume that closed formulas for the topological derivatives are available for the sake of this paper.
\end{remark}

 Throughout this work, we focus on minimization problems and remark that maximization problems can be treated analogously after a simple modification. Based on these definitions, we obtain the following notion of local optimality with respect to topological changes.

\begin{definition} \label{def_optimality}
    Given a shape function $\mathcal J$, an open domain $\Omega_1 \in \mathcal A$ is locally optimal for the minimization of $\mathcal J$ with respect to topological changes if 
    \begin{align*}
       \TD(x):= \chi_{\Omega_1}(x) \TD^{1 \rightarrow 2}(x) +\chi_{\Omega_2}(x) \TD^{2 \rightarrow 1}(x) > 0
    \end{align*}
    for all $x \in \Omega_1  \cup \Omega_2$.
\end{definition}

Note that topological derivatives and therefore also $\mathcal T$ are only defined in the interior of subdomains and not on boundaries or material interfaces.

\subsection{Level set algorithm solely based on topological derivative}

In a level set framework, a domain $\Omega_1$ is represented by means of a function $\psi: D \rightarrow \mathbb R$ such that
\begin{align} \label{eq_psi_x_t}
    \left\lbrace \;
    \begin{aligned}
        \psi(x) < 0 &\Longleftrightarrow x \in \Omega_1, \\
        \psi(x) = 0 &\Longleftrightarrow x \in \Gamma, \\
        \psi(x) > 0 &\Longleftrightarrow x \in D \setminus \overline{\Omega_1}=: \Omega_2.
    \end{aligned} \right.
\end{align}
Here, $\Gamma = \partial \Omega_1 \cap D$ includes the material interface between $\Omega_1$ and $\Omega_2$ and excludes the boundary of the hold-all domain $\partial D$.
In the algorithm introduced in \cite{AmstutzAndrae2006}, the evolution of the level set function $\psi$ is guided by the \textit{generalized topological derivative}, which is defined as
\begin{align} \label{eq_Gpsi}
    G(x) := \begin{cases}
                        -\TD^{1\rightarrow 2}(x), & x \in \Omega_1,\\
                        \TD^{2\rightarrow 1}(x), & x \in \Omega_2,
                   \end{cases}
\end{align}
for a fixed domain $\Omega_1$ and its complement $\Omega_2 = D \setminus \overline{\Omega_1}$,. When the domains $\Omega_1$ and $\Omega_2$ are represented by the level set function $\psi$ via \eqref{eq_psi_x_t}, we will indicate the dependence of the generalized topological derivative $G$ and of the topological derivatives $\TD^{i \rightarrow j}$ on the design by the additional subscript $\psi$, i.e., we will write $G_\psi$ and $\TD_{\psi}^{i \rightarrow j}$.

In this setting, we can state a sufficient condition for a domain $\Omega_1$ to satisfy the local optimality condition of Definition \ref{def_optimality}:
\begin{lemma} \label{lem_optiCond_2mat}
    Let $\mathcal J : \mathcal A \rightarrow \mathbb R$ a shape function and $\Omega_1 \in \mathcal A$ open. Let $\psi: D \rightarrow \mathbb R$ the level set function representing $\Omega_1$ according to \eqref{eq_psi_x_t}. Assume that $\Omega_1$ is such that there exists $c>0$ such that 
    \begin{align} \label{eq_psi_Gpsi}
        \psi(x) = c \, G_{\psi}(x)
    \end{align}
    for all $x \in D \setminus \partial \Omega_1$. Then $\Omega_1$ is locally optimal with respect to topological changes according to Definition \ref{def_optimality}.
\end{lemma}

\begin{proof}
    Let $x$ in $\Omega_1$ arbitrary, but fixed. Due to \eqref{eq_psi_x_t}, condition \eqref{eq_psi_Gpsi} and the definition in \eqref{eq_Gpsi}, we have
    \begin{align*}
        0> \psi(x) = c \, G_{\psi}(x) = -c\,\TD_\psi^{1 \rightarrow 2}(x),
    \end{align*}
    for some $c>0$. Thus we have that $\TD_\psi^{1\rightarrow 2}(x) > 0$. Similarly, we get for an arbitrary, but fixed point $\tilde x \in \Omega_2$ that
    \begin{align*}
        0< \psi(\tilde x) = c \, G_{\psi}(\tilde x) = c \,\TD_\psi^{2\rightarrow 1}(\tilde x),
    \end{align*}
    and thus $\TD_\psi^{2\rightarrow 1}(\tilde x) > 0$, which finishes the proof.
\end{proof}

The idea of the algorithm introduced in \cite{AmstutzAndrae2006} is to start with an initial design $\Omega_1^{(0)}$ represented by a level set function $\psi_0$ and to reach a setting where condition \eqref{eq_psi_Gpsi} holds by means of a fixed point iteration. For numerical stability reasons, this fixed point iteration is performed on the unit sphere $\mathcal S$ of the Hilbert space $L^2(D)$, $\mathcal S = \{ v \in L^2(D): \| v\|_{L^2(D)} = 1 \}$.

The algorithm reads as follows:

\begin{algorithm} \label{algo_2mat}
    Choose initial design $\psi_0 \in \mathcal S$ and $\varepsilon_\theta >0$ (e.g. $\varepsilon_\theta = 0.5 ^\circ$).\\
    For $k=0,1,2,\dots$
    \begin{enumerate}
        \item Compute $G_{\psi_k}$ according to \eqref{eq_Gpsi}.
        \item Compute $\theta_k:= \mbox{arccos} \left[ \left(\psi_k, \frac{ G_{\psi_k} }{ \|G_{\psi_k}\|_{L^2(D)}} \right)_{L^2(D)} \right]$.
        \item If $\theta_k <  \varepsilon_\theta$ then stop, \\
        else set 
        \begin{equation}  \label{eq_psi_2mat_update}
            \psi_{k+1} = \frac{1}{\mbox{sin} \, \theta_k} \left[ \mbox{sin}((1-\kappa_k)\theta_k) \psi_k + \mbox{sin}(\kappa_k \theta_k) \frac{G_{\psi_k} }{ \| G_{\psi_k}\|_{L^2(D)}} \right], 
        \end{equation}
        where $\kappa_k =  \mbox{max} \{1, \frac{1}{2}, \frac{1}{4}, \dots \}$ such that $\mathcal J\left(\Omega_1^{(k+1)} \right) < \mathcal J\left(\Omega_1^{(k)}\right)$.
    \end{enumerate}
\end{algorithm}

In step 2 of Algorithm \ref{algo_2mat}, the angle $\theta_k$ in an $L^2(D)$-sense between the current level set function $\psi_k$ and a scaled version of the generalized topological derivative $G_{\psi_k}$ is computed. Note that if $\theta_k=0$ then condition \eqref{eq_psi_Gpsi} is satisfied and the corresponding domain $\Omega_1^{(k)}$ is locally optimal according to Definition \ref{def_optimality}. Also note that, by construction, we get that $\psi_k \in \mathcal S$ for all $k$. The sets $\Omega_1^{(k+1)}$ and $\Omega_1^{(k)}$ denote the shapes represented by the corresponding level set functions $\psi_{k+1}$ and $\psi_k$, respectively. For more details on the algorithm, we refer the reader to \cite{AmstutzAndrae2006, Amstutz2011a}.

As it is emphasized in \cite[Sec. 3.3]{AmstutzAndrae2006}, the algorithm is evolving along a descent direction, which is an important ingredient for it to be successful. This means that a local change of material around a point $\hat x \in \Omega_1 \cup \Omega_2$ from iteration $k$ of the algorithm to the next iteration $k+1$ can only happen if the corresponding topological derivative at iteration $k$ was negative. This observation can be summarized as follows:
\begin{lemma} \label{lem_descentDir_2mat}
    Let $\psi_k$ and $\psi_{k+1}$ be two subsequent iterates obtained by Algorithm \ref{algo_2mat}. Then we have for $\hat x \in D$ that
    \begin{enumerate}[(i)]
     \item $\psi_k(\hat x) < 0 < \psi_{k+1}(\hat x) \quad \Longrightarrow \quad \TD_{\psi_k}^{1 \rightarrow 2}(\hat x) < 0$,
     \item $\psi_k(\hat x) > 0 > \psi_{k+1}(\hat x) \quad \Longrightarrow \quad \TD_{\psi_k}^{2 \rightarrow 1}(\hat x) < 0$.
    \end{enumerate}
\end{lemma}
\begin{proof}
    Since $\mbox{sin}(\theta)>0$ and $\mbox{sin}(s \, \theta)>0$ for all $\theta \in (0, \pi)$ and all $s \in (0,1)$, we conclude from $\psi_k(\hat x) < 0 < \psi_{k+1}(\hat x)$ and \eqref{eq_psi_2mat_update} that $ G_{\psi_k}(\hat x) > 0$ and thus, since $\hat x \in \Omega_1^{(k)}$, it follows that $\TD_{\psi_k}^{1\rightarrow 2} (\hat x) < 0$. The second statement follows analogously.
\end{proof}
We will later establish a similar result in the case of multiple materials in Section \ref{sec_algo_m_mat}.

We make the following important remark concerning the numerical realization of Algorithm \ref{algo_2mat} in the PDE-constrained case, i.e., in the case where the objective function $\mathcal J$ depends on the shape $\Omega$ via the solution $u$ of a constraining boundary value problem. We restrict our discussion to the case of problems posed in an $H^1(D)$ setting such as linearized elasticity or Laplace-type problems. Problems where the solution $u$ is not continuous across element boundaries like problems posed in $H(curl,D)$ deserve a separate investigation.
The following aspects are discussed in \cite{AmstutzAndrae2006} for the case of linearized elasticity:
\begin{remark} \label{rem_interface}
    We choose a finite-dimensional space $V_h$ of $L^2(D)$ endowed with the $L^2(D)$ inner product to represent the design variable $\psi$. To guarantee the smoothness of the design, the level set function $\psi$ should be continuous across the interface $\partial \Omega_1$. Therefore, we choose $V_h$ to be the space of globally continuous and piecewise linear functions on a given triangulation of $D$. As proposed in \cite{AmstutzAndrae2006}, we will use the same mesh and the same finite elements for representing the design and for solving the PDE constraint. In elements that are cut by the material interface $\Gamma = \{x:\psi(x) = 0\}$, the material parameters are averaged by linear interpolation. Denoting the material parameters in $\Omega_i$ by $\alpha_i$, $i=1,2$, the material parameter in element $T$ is computed by
    $$ \alpha |_T = \frac{|T \cap \Omega_1|}{|T|} \alpha_1 + \frac{|T \cap \Omega_2|}{|T|} \alpha_2.$$
    Using piecewise linear finite elements, the resulting topological derivatives are often piecewise constant functions. In order to get a representation of the topological derivative in $V_h$, which is necessary for the numerical realization of \eqref{eq_psi_2mat_update}, we compute an average of the topological derivative around each node of the mesh. This gives us a function $\tilde G \in V_h$. Note that this latter procedure can also be seen as a sensitivity filtering technique. Such techniques are well-known in density-based topology optimization as regularization methods.
\end{remark}

\section{Multi-material topology optimization using topological derivatives}
\label{sec_algo_m_mat}

In this section, we introduce a generalization of Algorithm \ref{algo_2mat} to the case of an arbitrary number of materials. We will represent the design consisting of $M$ materials by means of a vector-valued level-set function $\psi$ mapping from the hold-all domain $D$ into $\mathbb R^{M-1}$. The algorithm will be very similar to Algorithm \ref{algo_2mat} with an appropriately chosen generalized topological derivative $ G$.

\subsection{Topological derivatives in multi-material setting}
We consider again an open and bounded hold-all domain $D \subset \mathbb R^d$ and a set of admissible subsets of $D$, $\mathcal A \subset \mathcal P(D)$. For $M>2$ let
\begin{align*}
    \mathcal A_M := \Biggl\{ &(\Omega_1, \dots,\Omega_M) \in \mathcal A \times \dots \times \mathcal A: \, \overline D = \bigcup_{l=1}^M \overline \Omega_l, \;  \Omega_i \text{ open}  \\
    & \text{ and } \Omega_i \cap \Omega_j = \emptyset \text{ for } i,j \in \{1,\dots,M\}, i\neq j \Biggr\}.
\end{align*}
A multi-material shape function is a mapping
\begin{align*}
    \mathcal J: \mathcal A_M &\rightarrow \mathbb R \\
    (\Omega_1, \dots, \Omega_M) & \mapsto \mathcal J(\Omega_1, \dots, \Omega_M).
\end{align*}
Similarly to Definition \ref{def_TD}, we define the topological derivatives of a multi-material shape function $\mathcal J$ for $(\Omega_1, \dots, \Omega_M)$ at a point $z \in \Omega_i$, $i \in \{1,\dots, M\}$:

\begin{definition} \label{def_TD_mMat}
    Let $\mathcal J : \mathcal A_M \rightarrow \mathbb R$ a multi-material shape function and $(\Omega_1,\dots, \Omega_M) \in \mathcal A_M$.
    Let $\omega \subset \mathbb R^d$ open with $0 \in \omega$ and $i,j \in \{1, \dots, M\}$ fixed with $i \neq j$. Let $z \in \Omega_i$ and, for $\varepsilon>0$ small, let $\omega_\varepsilon := z + \varepsilon \omega$. Then, the topological derivative at the point $z \in \Omega_i$ with respect to $\Omega_j$ is defined as the limit
\begin{align*}
	\TD^{i \rightarrow j}(z) := \underset{\varepsilon \searrow 0}{\mbox{lim }} \frac{\mathcal J(\Omega_1, \dots, \Omega_i \setminus \omega_\varepsilon, \dots, \Omega_j \cup \omega_\varepsilon, \dots, \Omega_M) - \mathcal J(\Omega_1, \dots, \Omega_M) }{ \ell(\varepsilon) }
\end{align*}
where $\ell : \mathbb R^+ \rightarrow \mathbb R^+$ is a continuous function satisfying $\ell(\varepsilon) \rightarrow 0$ as $\varepsilon \rightarrow 0$.

Moreover, for all $x \in \Omega_i$, we define the vector of topological derivatives
\begin{align} \label{eq_vecTDs}
    \TD^{(i)}(x):=\left( \begin{array}{c}
                                \TD^{i \rightarrow 1}(x)\\
                                \dots\\
                                \TD^{i \rightarrow i-1}(x)\\
                                \TD^{i \rightarrow i+1}(x)\\
                                \dots\\
                                \TD^{i \rightarrow M}(x)
                                        \end{array} \right) \in \mathbb R^{M-1},
\end{align}
and for all $x \in D \setminus \left( \bigcup_{j=1}^M \partial \Omega_j \right)$ we define
\begin{align*}
    \TD(x):= \sum_{i=1}^M \chi_{\Omega_i}(x) \TD^{(i)}(x).
\end{align*}

\end{definition}
 Again, we obtain a notion of local optimality:
 
 \begin{definition} \label{def_locOpt_mMat}
    Let a multi-material shape function $\mathcal J: \mathcal A_M \rightarrow \mathbb R$ be given. A tuple of domains $(\Omega_1, \dots, \Omega_M) \in \mathcal A_M$ is locally optimal for the minimization of $\mathcal J$ with respect to topological changes if
    \begin{equation} \label{eq_locOpt_mMat}
         \TD(x) > \mathbf 0 \in \mathbb R^{M-1}
    \end{equation}
    for all $x \in D \setminus \left( \cup_{j=1}^M \partial \Omega_j \right)$.
 \end{definition}
 Here, $\mathbf 0$ represents the zero vector in $\mathbb R^{M-1}$ and the inequality in \eqref{eq_locOpt_mMat} is meant componentwise.

\subsection{Domain representation}
We represent a design $(\Omega_1, \dots, \Omega_M) \in \mathcal A_M$ by means of a vector-valued level set function $\psi : D \rightarrow \mathbb R^{M-1}$. Similarly to the two-material case, we will divide the image space of $\psi$, $\mathbb R^{M-1}$, into $M$ different convex, open sectors $S_1,\dots ,S_M$ such that $\mathbb R^{M-1} = \bigcup_{l=1}^M \overline S_l$. Each sector $S_l$, $l =1,\dots, M$, is uniquely determined by $M-1$ hyperplanes $ H_{l,1},\dots,  H_{l,l-1},  H_{l,l+1},  H_{l,M}$. Each hyperplane $H_{i,j}$ is uniquely determined by its normal vector $n^{i \rightarrow j} \in \mathbb R^{M-1}$ which is oriented such that $n^{i \rightarrow j}|_{\overline S_i \cap \overline S_j}$ is pointing out of sector $S_i$ and into sector $S_j$, see Figure \ref{fig_sectors} for an illustration in the case $M=3$.

\begin{definition} \label{def_matrixNormals}
    For $l \in \{1,\dots,M\}$, we define the matrix of normal vectors pointing into sector $S_l$ by
	\begin{equation*}
		N^{(l)} := \left( \begin{array}{cc}
					(n^{1\rightarrow l})^\top \\
					\dots \\
					(n^{l-1\rightarrow l})^\top \\
					(n^{l+1\rightarrow l})^\top \\
					\dots\\
					(n^{M\rightarrow l})^\top \\
		                 \end{array}
		                 \right) \in \mathbb R^{M-1,M-1}.
	\end{equation*}
\end{definition}

\begin{assumption} \label{assump_Ninvert}
    We assume that the sectors $S_l$, $l =1,\dots, M$,
    \begin{enumerate}[(a)]
     \item are convex and 
     \item are chosen in such a way that $N^{(l)}$ is invertible for all $l \in \{1,\dots, M\}$.
    \end{enumerate}
\end{assumption}

\begin{remark}
    In \cite{WangCMAME2015} and several other publications, the sectors are defined just by the signs of otherwise independent level set functions. As an example, when $M=3$, in that context the three sectors would be given as $S_1 =\{(\psi_1,\psi_2):\psi_1<0 \}$, $S_2=\{(\psi_1, \psi_2): \psi_1>0 \wedge \psi_2 <0 \}$, $S_3=\{(\psi_1, \psi_2): \psi_1>0 \wedge \psi_2 >0 \}$. Such a configuration is excluded by virtue of Assumption \ref{assump_Ninvert}(b) since the normal vectors $n^{2 \rightarrow 1}$ and $n^{3\rightarrow 1}$ would be linearly dependent and therefore the matrix $N^{(1)}$ would not be invertible in this setting. In our approach, it is important that the angles between the hyperplanes are smaller than 180$^\circ$.
\end{remark}

For the rest of this paper, we assume that Assumption~\ref{assump_Ninvert} is satisfied. In this setting, it holds that every sector $S_l$, $l=1\dots,M$, can be written as
\begin{equation} \label{char_S}
	S_l = \{y \in \mathbb R^{M-1}: y \cdot n^{j\rightarrow l}>0 \; \forall j \neq l \}.
\end{equation}

\begin{figure}
    \begin{center}
        \includegraphics[width=.5\textwidth]{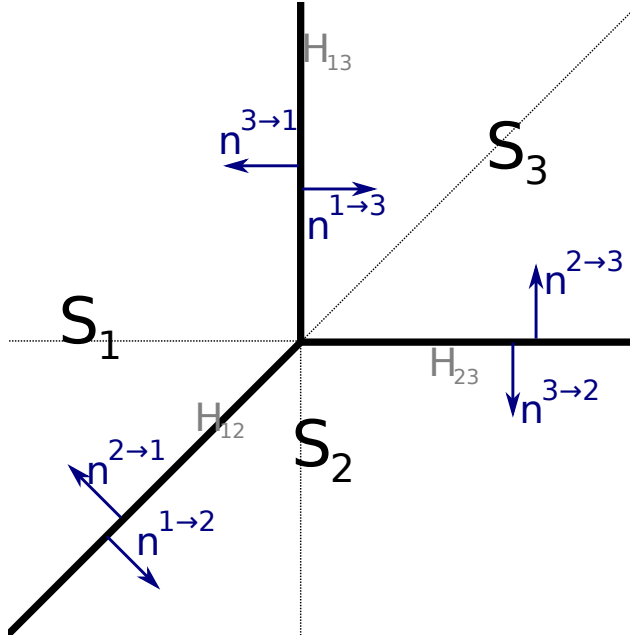}
    \end{center}
    \caption{Illustration of sectors in case $M=3$.}
    \label{fig_sectors}
\end{figure}

The subdomains $\Omega_l$, $l\in\{1,\dots,M \}$ are represented by means of a vector-valued level set function $\psi : D \rightarrow \mathbb R^{M-1}$ as
\begin{align} \label{def_psi_Mmat}
	 x \in \Omega_l \Longleftrightarrow \psi(x) \in S_l
\end{align}
 for all $l \in \{1,\dots,M\}$.

\subsection{Multi-material level set algorithm}
The main challenge in generalizing Algorithm \ref{algo_2mat} to the multi-material case is to define a generalized topological derivative such that, on the one hand, an optimality condition similar to the one in Lemma \ref{lem_optiCond_2mat} holds, and on the other hand the resulting algorithm is evolving along a descent direction, similarly to the observation made in Lemma \ref{lem_descentDir_2mat}.

We define the generalized topological derivative in the following way:
\begin{definition} \label{def_genTD_Mmat}
	For $l=1,\dots, M$ and  $x \in \Omega_l$ we define $G_l:\Omega_l \rightarrow \mathbb R^{M-1}$,
	\begin{align}
		G_l(x) := \left(N^{(l)}\right)^{-1} \TD^{(l)}(x),
	\end{align}
	with $N^{(l)}$ given by Definition \ref{def_matrixNormals} and $\TD^{(l)}$ defined in \eqref{eq_vecTDs}. Moreover, we define the generalized topological derivative $G:D \setminus \left( \bigcup_{j=1}^M \partial \Omega_j \right) \rightarrow \mathbb R^{M-1}$,
	\begin{equation} \label{eq_genTD_mMat}
		G(x):=  \sum_{l=1}^M \chi_{\Omega_l}(x) G_l(x).
	\end{equation}	
\end{definition}
Again note that $G$ is only defined on the union of the interior of the subdomains. We will again indicate the generalized topological derivative for a design represented by a vector-valued level set function $\psi$ by adding a subindex, i.e. by writing $G_\psi$.

\begin{lemma} \label{lem_Gn}
	Let the generalized topological derivative $G:D\setminus \left( \bigcup_{j=1}^M \partial \Omega_j \right)  \rightarrow \mathbb R^{M-1}$ be defined as in Definition \ref{def_genTD_Mmat}. Then it holds for all $l,i \in \{1,\dots,M\}$, $i\neq l$, and all $x \in \Omega_l$
	\begin{equation}
		G|_{\Omega_l}(x) \cdot n^{i\rightarrow l} = \TD^{l \rightarrow i}(x).
	\end{equation}
\end{lemma}
\begin{proof}
	By definition, we have that
	\begin{align*}
		G|_{\Omega_l}(x) \cdot n^{i\rightarrow l} &= G_l(x) \cdot n^{i\rightarrow l} = \left[\left(N^{(l)}\right)^{-1} \TD^{(l)}(x)\right] \cdot n^{i\rightarrow l} \\
		&= \left(\TD^{(l)}(x) \right)^\top \, \left(\left(N^{(l)}\right)^{\top}\right)^{-1} \, n^{i\rightarrow l}.
	\end{align*}
	Since the vector $n^{i\rightarrow l}$ appears as a column in the matrix $\left(N^{(l)}\right)^{\top}$, we get that
	\begin{align*}
		G|_{\Omega_l}(x) \cdot n^{i\rightarrow l} &= \left(\TD^{(l)}(x) \right)^\top e_k 
	\end{align*}
	where $e_k$ is the $k$-th unit vector in $\mathbb R^{M-1}$ and $k$ is the column index of $n^{i\rightarrow l}$ in $\left(N^{(l)}\right)^{\top}$, i.e., $k=i$ if $i<l$ and $k=i-1$ if $i>l$. On the other hand, the $k$-th element of the vector $\TD^{(l)}(x)$ is just $\TD^{l\rightarrow i}(x)$ which finishes the proof.
\end{proof}

Now we are able to show the following optimality condition:
\begin{theorem} \label{theo_optimality}
	Let $(\Omega_1, \dots, \Omega_M) \in \mathcal A_M$. Let $\psi : D \rightarrow \mathbb R^{M-1}$ be a vector-valued level set function such that \eqref{def_psi_Mmat} holds and let $G_\psi$ the generalized topological derivative according to Definition \ref{def_genTD_Mmat} for the configuration given by $\psi$. Suppose that there exists a positive constant $c>0$ such that 
	\begin{align} \label{eq_optiCond_Mmat}
        \psi(x) = c \, G_\psi(x)
	\end{align}
	for all $x \in \bigcup_{l=1}^M \Omega_l$. Then, $(\Omega_1, \dots, \Omega_M)$ is locally optimal according to Definition \ref{def_locOpt_mMat}.
	
\end{theorem}
\begin{proof}
	Suppose that $\psi = G_\psi$ and take $\hat x \in \Omega_l$ for an arbitrary, but fixed $l \in \{1,\dots, M\}$. Then $\psi(\hat x) \in S_l$, i.e.,
		\begin{align*}
			\psi(\hat x) \cdot n^{i \rightarrow l}  > 0 \quad \mbox{for all } i\neq l,
		\end{align*}
		due to \eqref{char_S}. Thus, since $\hat x \in \Omega_l$, we have by assumption \eqref{eq_optiCond_Mmat} and due to Lemma~\ref{lem_Gn}
		\begin{align*}
			0<\psi(\hat x) \cdot n^{i \rightarrow l} = c \, G_\psi(\hat x) \cdot n^{i \rightarrow l}  =c \,G_\psi|_{\Omega_l}(\hat x) \cdot n^{i \rightarrow l} = c \,\TD^{l\rightarrow i}(\hat x)  \quad \mbox{for all } i\neq l.
		\end{align*}
	Since $c$ is positive and $l$ was arbitrary, the statement follows.
\end{proof}

Theorem \ref{theo_optimality} states that $\psi(x) = c \,G_\psi(x)$ with $c>0$ is a local sufficient optimality condition meaning that switching the material in only a small neighborhood of one point to any other material will yield an increase of the objective function. Analogously to the case of two materials, the idea of the algorithm is again to reach a design which satisfies \eqref{eq_optiCond_Mmat}. Again, this is achieved by a fixed point iteration on the unit sphere of a Hilbert space $X$ where the angle $\theta$ between the level set function $\psi$ and the corresponding generalized topological derivative $G_\psi$ is driven to zero. Here, we choose $X = L^2(D, \mathbb R^{M-1})$.

\begin{algorithm} \label{algo_TDLSMmat}
    Choose initial design $\psi_0 \in \mathcal S$ and $\varepsilon_\theta >0$ (e.g. $\varepsilon_\theta = 0.5 ^\circ$).\\
    For $k=0,1,2,\dots$
    \begin{enumerate}
        \item Compute $ G_{\psi_k}$ according to \eqref{eq_genTD_mMat}.
        \item Compute $\theta_k:= \mbox{arccos} \left[ \left(\psi_k, \frac{ G_{\psi_k} }{ \| G_{\psi_k}\|_{X}} \right)_{X} \right]$.
        \item If $\theta_k <  \varepsilon_\theta$ then stop, \\
        else set 
        \begin{equation}  \label{eq_psi_Mmat_update}
            \psi_{k+1} = \frac{1}{\mbox{sin} \, \theta_k} \left[ \mbox{sin}((1-\kappa_k)\theta_k) \psi_k + \mbox{sin}(\kappa_k \theta_k) \frac{ G_{\psi_k} }{ \| G_{\psi_k}\|_{X}} \right], 
        \end{equation}
        where $\kappa_k = \mbox{max} \{1, \frac{1}{2}, \frac{1}{4}, \dots \}$ such that $\mathcal J\left(\psi_{k+1} \right) < \mathcal J\left( \psi_k \right)$.
    \end{enumerate}
\end{algorithm}
Here we used the notation $\mathcal J(\psi_k)$ for the objective value for the design $(\Omega_1^{(k)}, \dots, \Omega_M^{(k)} )$ represented by the vector-valued level set function $\psi_k$.

Similarly as in the case of two materials, we can show that, when choosing the generalized topological derivative $G$ according to Definition \ref{def_genTD_Mmat}, Algorithm \ref{algo_TDLSMmat} is evolving along a descent direction:
\begin{theorem} \label{descentDirMmat}
	Let $\psi_k, \psi_{k+1} : D \rightarrow \mathbb R^{M-1}$ the vector-valued level set functions representing the designs at iterations $k$ and $k+1$ of Algorithm \ref{algo_TDLSMmat}, respectively. 	Then, for all $i,j \in \{1,\dots,M\}, i\neq j$, and all $\hat x \in D$, it holds:
	\begin{align*}
		\psi_k(\hat x) \in S_i \wedge \psi_{k+1}(\hat x) \in S_j \Longrightarrow \TD_{\psi_k}^{i\rightarrow j}(\hat x) <0.
	\end{align*}
\end{theorem}
Here, $\TD_{\psi_k}^{i\rightarrow j}(\hat x)$ represents the topological derivative for the design given by $\psi_k$ for switching from material $i$ to material $j$ around point $\hat x$.
\begin{proof}
	We have $\psi_k(\hat x) \in S_i$ and $\psi_{k+1}(\hat x) \in S_j$ where $i,j \in \{1,\dots,M \}, i\neq j$.
	From \eqref{eq_psi_Mmat_update}, we see that
        \begin{align} \label{eq_intermedres}
			\mu \frac{G_{\psi_k}(\hat x)}{\|G_{\psi_k}(\hat x)\|} = \psi_{k+1}(\hat x)  - \lambda \psi_k(\hat x).
		\end{align}
		with $\lambda, \mu >0$. Since $-n^{j\rightarrow i} = n^{i\rightarrow j}$ it holds that
		\begin{align*}
			(\psi_{k+1}(\hat x)  - \lambda \psi_k(\hat x))\cdot (- n^{j\rightarrow i}) = \psi_{k+1}(\hat x)\cdot n^{i\rightarrow j} +\lambda \psi_k(\hat x) \cdot n^{j\rightarrow i} > 0
		\end{align*}
		because of \eqref{char_S} since $\psi_{k+1}(\hat x) \in S_j$ and $\psi_{k}(\hat x) \in S_i$. Thus, due to \eqref{eq_intermedres} and $\hat x \in \Omega_i$, we have that also
		\begin{align*}
			0< G_{\psi_k}(\hat x)  \cdot (- n^{j\rightarrow i})= - G_{\psi_k}|_{\Omega_i}(\hat x)  \cdot n^{j\rightarrow i} = - \TD_{\psi_k}^{i\rightarrow j}(\hat x)
		\end{align*}
		due to Lemma \ref{lem_Gn} and thus $\TD_{\psi_k}^{i\rightarrow j}(\hat x)<0$.

\end{proof}

Theorem \ref{descentDirMmat} says that, when the design is switched from one material to another at a certain point in the domain in the course of Algorithm \ref{algo_TDLSMmat}, then the corresponding topological derivative at iteration $k$ was negative. Thus, the algorithm is somehow evolving along a descent direction, which is an important ingredient for the algorithm to reach the optimality condition of Theorem \ref{theo_optimality}.

\section{Numerical Experiments} \label{sec_numerics}
We will illustrate the flexibility and the potential of Algorithm \ref{algo_TDLSMmat} in two classes of examples. On the one hand, we consider classical applications of minimizing the compliance in the framework of linearized elasticity. On the other hand, we show an academic example where the exact solution, which consists of $M=8$ phases, is known. Before showing numerical results, we will discuss how to obtain the necessary data structure for an arbitrary number of materials $M$.

\subsection{Creating data structure in arbitrary dimensions}
For the practical realization of Algorithm \ref{algo_TDLSMmat} with an arbitrary number of materials $M$, it is important to have access to the normal vector matrices $N^{(l)}$, $l=1,\dots, M$, introduced in Definition \ref{def_matrixNormals}, which define the sectors $S_1, \dots, S_M$ of $\mathbb R^{M-1}$. In particular, it is important to ensure that each of the matrices $N^{(l)}$ is invertible (cf. Assumption \ref{assump_Ninvert}). This is equivalent to making sure that, for each $l \in \{1,\dots, M\}$, the normal vectors $n^{i\rightarrow l}$, $i\neq l$, pointing into sector $S_l$ are linearly independent. We emphasize that this data has to be created only once for any fixed $M$. For that reason, we provide a MATLAB implementation of this part of the code in \ref{appendix_code}. Since the code has to be run only once for any fixed $M$ and the computational time is negligible for a moderate number of materials $M$, we put an emphasis on readability of the code over computational efficiency.

Recall that each of the sectors $S_l$ is defined by $M-1$ hyperplanes in $\mathbb R^{M-1}$. Each of the hyperplanes in $\mathbb R^{M-1}$ can be defined by $M-2$ linear independent vectors. We choose these vectors to be defined by two points: the origin $\mathbf 0 \in \mathbb R^{M-1}$ and one other point. The structure is fully determined by the choice of these points. The procedure reads as follows:
\begin{enumerate}[(1)]
    \item Define $P_0 = \mathbf 0$ the origin in $\mathbb R^{M-1}$ and choose $P_1, \dots, P_M \in \mathbb R^{M-1}$. These points have to be chosen in such a way that Assumption \ref{assump_Ninvert} is satisfied. This holds, e.g., for $P_i = e_i$ the $i$-th unit vector in $\mathbb R^{M-1}$ for $i=1,\dots,M-1$ and $P_M = (-1, \dots, -1) \in \mathbb R^{M-1}$.
    \item Define $\binom{M}{M-2} = M(M-1)/2$ many hyperplanes as the span of all combinations of $M-2$ vectors out of $M$ vectors of the form $(P_i - P_0)$, $i \in {1 , \dots, M}$. 
    \item Define $M$ sectors $S_1, \dots, S_M \subset \mathbb R^{M-1}$ where each sector is bounded by a certain combination of $M-1$ hyperplanes.
    \item For $i, j \in \{1,\dots, M\}$, define $n^{i \rightarrow j}$ as the unit normal vector of the hyperplane separating sector $S_i$ from sector $S_j$, oriented in such a way that $n^{i \rightarrow j}$ points out of $S_i$ and into $S_j$. It holds $n^{j \rightarrow i}=-n^{i \rightarrow j}$.
\end{enumerate}

\begin{center}
     \begin{table}
        \begin{tabular}{ccc}
        \textbf{M=3} & & \textbf{M=4} \\
        \begin{tabular}{c|c|| c}
                $i$ & $j$ & $n^{i \rightarrow j}$  \\ \hline\hline
                $1$ & $2$& $(\sqrt{2}/2,-\sqrt{2}/2)^\top$ \\
                $1$ & $3$& $(1,0)^\top$\\
                $2$ & $3$& $(0,1)^\top$
            \end{tabular}
            & \qquad \qquad &
            
            \begin{tabular}{c|c|| c}
                $i$ & $j$ & $n^{i \rightarrow j}$  \\ \hline\hline
                $1$ & $2$& $(\sqrt{2}/2,-\sqrt{2}/2, 0)^\top$ \\
                $1$ & $3$& $(\sqrt{2}/2,0,-\sqrt{2}/2)^\top$\\
                $1$ & $4$& $(1,0,0)^\top$\\
                $2$ & $3$& $(0, \sqrt{2}/2,-\sqrt{2}/2)^\top$ \\
                $2$ & $4$& $(0,1,0)^\top$\\
                $3$ & $4$& $(0,0,1)^\top$\\
            \end{tabular}
        \end{tabular}
        \caption{Examples of normal vectors $n^{i\rightarrow j}$ for $M=3$ and $M=4$; results obtained by call of function \texttt{createDataStructure(M)} and subsequently \texttt{getNormal(...)}, see \ref{appendix_code}.}
        \label{tab_data_struc_M3_M4}
    \end{table}
\end{center}

Examples of the involved normal vectors for $M=3$ and $M=4$ can be found in Table \ref{tab_data_struc_M3_M4}. Recall that the inverted normal vectors $n^{j \rightarrow i}$ can be obtained as $n^{j \rightarrow i} =- n^{j \rightarrow i}$. The function \texttt{createDataStructure(...)} computes the information needed for all normal vectors by following steps (1)--(4) from above. The normal can be extracted by a call of \texttt{getNormal(...)}, and the function \texttt{isInSector(...)} can be used to check whether a point $p \in \mathbb R^{M-1}$ is in a specific sector of $\mathbb R^{M-1}$. All of these methods are available in \ref{appendix_code}.

\subsection{Application to structural optimization}
We illustrate the usage of Algorithm \ref{algo_TDLSMmat} for three classical problems from structural optimization where we minimize the compliance of a structure. We search for the optimal distribution of $M=3$ different isotropic materials (a strong material, a weak material and void) within a given domain $D \subset \mathbb R^2$. For the three phases, we use Young's moduli $E_1 = 1$ (strong), $E_2 = 0.5$ (weak), $E_3 = 10^{-4}$ (void) and Poisson's ratios $\nu_1 = \nu_2 = 0.3333$ and $\nu_3= 0.3333 *10^{-4}$. In addition we impose a bound on the allowed volumes in an implicit way, by adding the volumes of the strong material $|\Omega_1|$ and the weak material $|\Omega_2|$, weighted by factors $\ell_1 = 2$ and $\ell_2 = 0.5$, respectively, to the objective function. We solve the minimal compliance problem constrained by the PDE system of linearized elasticity in the plane stress setting,
\begin{subequations}
    \begin{align}
        \underset{(\Omega_1, \Omega_2,\Omega_3) \in \mathcal A_3}{\mbox{min }} \mathcal J(\Omega_1,\Omega_2, \Omega_3) = \mathcal C(u) &+ \ell_1 |\Omega_1| + \ell_2 |\Omega_2| \label{eq_functionalComp}\\
        \mbox{s.t. }  - \mbox{div}(A_{\Omega_1, \Omega_2, \Omega_3} \epsilon(u)) &= 0  \quad\mbox{in } D , \label{eq_pde_1}\\
        A_{\Omega_1, \Omega_2, \Omega_3} \epsilon(u) n  &= g \quad \mbox{on } \Gamma_N,\label{eq_pde_2}\\
        u &= 0 \quad\mbox{on } \Gamma_D,\label{eq_pde_3}\\
        A_{\Omega_1, \Omega_2, \Omega_3} \epsilon(u) n  &= 0 \quad \mbox{on } \partial D \setminus (\Gamma_N \cup \Gamma_D),\label{eq_pde_4}
    \end{align}
\end{subequations}
on the domain $D \subset \mathbb R^2$. Here, $\epsilon(u)=\frac{1}{2}(\nabla u + \nabla u^\top)$ represents the symmetric gradient, $\mathcal C(u) = \int_D A_{\Omega_1, \Omega_2, \Omega_3} \epsilon(u) : \epsilon(u) \mbox{d}x$ represents the compliance of the structure and $A_{\Omega_1, \Omega_2, \Omega_3}(x)= \chi_{\Omega_1}(x) A_1 +  \chi_{\Omega_2}(x) A_2 + \chi_{\Omega_3}(x) A_3$ the elasticity tensor. The elasticity tensors $A_i$, $i=1,2$, of strong and weak material are determined by the materials' Young's moduli $E_i$ and Poisson's ratios $\nu_i$ as
\begin{equation*}
    A_i e = 2 \mu_i e + \lambda_i \mbox{tr}(e) I, \; e \in \mathbb R^{2\times 2},
\end{equation*}
with the Lam\'e parameters $\lambda_i = \frac{E_i}{2(1+\nu_i)}$ and $\mu_i = \frac{\nu_i E_i}{2(1+\nu_i)(1-\nu_i)}$. The elasticity tensor for the void region is chosen as $A_3 = 10^{-4} (A_1 + A_2)$.

For any point $z \in \Omega_i$ and $i \in \{1,\dots,M\}$, the topological derivative for introducing an inclusion of material $j \neq i$ around $z$ reads (see e.g. \cite{GiustiFerrerOliver2016})
\begin{align} \label{eq_TDij_linEl}
    \TD^{i \rightarrow j}(z) = (\mathbb P^{i \rightarrow j} A_i \epsilon(u(z))) : \epsilon(u(z)) - \ell_i + \ell_j,
\end{align}
where $\ell_3 = 0$. Here, we used the Frobenius inner product defined by $A:B = \sum_{i,j=1}^n A_{ij}B_{ij}$ for $A,B \in \mathbb R^{n \times n}$. The fourth order polarization tensor $\mathbb P^{i \rightarrow j}$ is given by
\begin{align*}
    \mathbb P^{i \rightarrow j} =  \frac{1}{\beta \gamma + \tau_1} \left[ (1+\beta)(\tau_1-\gamma) \mathbb I + \frac{1}{2} (\alpha - \beta) \frac{\gamma(\gamma-2\tau_3) + \tau_1 \tau_2}{\alpha \gamma + \tau_2} (I \otimes I) \right]
\end{align*}
with the $(i,j)$-dependent coefficients
\begin{align*}
    &\alpha = \alpha^{(i)} = \frac{1+\nu_i}{1-\nu_i},& \; 
    &\beta = \beta^{(i)} = \frac{3-\nu_i}{1+\nu_i},& \; 
    &\gamma = \gamma^{(i \rightarrow j)} = \frac{E_j}{E_i},& \\
    &\tau_1 = \tau_1^{(i \rightarrow j)} = \frac{1+\nu_j}{1+\nu_i},& \;
    &\tau_2 = \tau_2^{(i \rightarrow j)} = \frac{1-\nu_j}{1-\nu_i},& \;
    &\tau_3 = \tau_3^{(i \rightarrow j)} = \frac{\nu_j (3 \nu_i-4)+1}{\nu_i (3 \nu_i-4)+1}.
\end{align*}
Here, $\mathbb I$ denotes the identity tensor of order 4 and $I$ the usual identity matrix in $\mathbb R^2$.
In the following, we will apply Algorithm \ref{algo_TDLSMmat} to three different settings, which are shown in Figure \ref{fig_settings}, where also the boundary regions $\Gamma_D$ with zero displacement and the region $\Gamma_N$ where a given force is acting, are indicated.
For the solution of the constraining boundary value problem \eqref{eq_pde_1}--\eqref{eq_pde_4}, we employed the finite element method using piecewise linear, globally continuous ansatz and test functions on a triangular mesh. The implementation is done in FreeFem++ \cite{Hecht2012} extending a code that was kindly provided by Dr. Charles Dapogny \cite{DapognyCode2018}.

In the course of the algorithm, the material interfaces are evolving over a fixed grid. We adopt the procedure suggested in the two-material case in \cite{AmstutzAndrae2006} and discussed here in Remark \ref{rem_interface} to the case of three materials as follows: On the one hand, the elasiticity tensor in the elements $T$ that are cut by a material interface is obtained by linear interpolation between all three materials,
\begin{equation} \label{eq_matInterpol_3mat}  A|_T = A_1 \frac{|T \cap \Omega_1|}{|T|} + A_2 \frac{|T \cap \Omega_2|}{|T|} + A_3 \frac{|T \cap \Omega_3|}{|T|}.\end{equation}
On the other hand, the topological sensitivities, which in the framework of piecewise linear finite elements are given as piecewise constant on the mesh, are interpolated the other way around, from the centers of the cells to the vertices. This is done by assigning the average of the sensitivities in the cells surrounding a mesh node $x_j$ to the new, filtered sensitivity at this node $x_j$. It is well-known that this kind of sensitivity filtering has a regularizing effect and avoids the formation of checkerboard patterns \cite{MauteSigmund2013}.

In numerical experiments, we observed (both in the multi-material case as well as in the case of only two materials) that, when choosing an unstructured triangular mesh, the non-symmetry of the mesh can be critical and lead to non-symmetric optimal designs when the step size parameter $\kappa$ in Algorithm \ref{algo_TDLSMmat} is chosen too large. For this reason, we set an upper bound $\overline \kappa$ on the maximum stepsize and choose $\kappa_k$ in iteration $k$ as $\kappa_k = \overline \kappa \, \mbox{max} \,  \{1, \frac{1}{2}, \frac{1}{4}, \dots \}$ such that a descent is achieved, cf. step 3 of Algorithm \ref{algo_TDLSMmat}. In all of the presented examples, we ran 500 iterations of the algorithm.
\begin{figure}
    \begin{tabular}{ccccc}
        \begin{minipage}{0.3\textwidth}\centering \includegraphics[width=\textwidth]{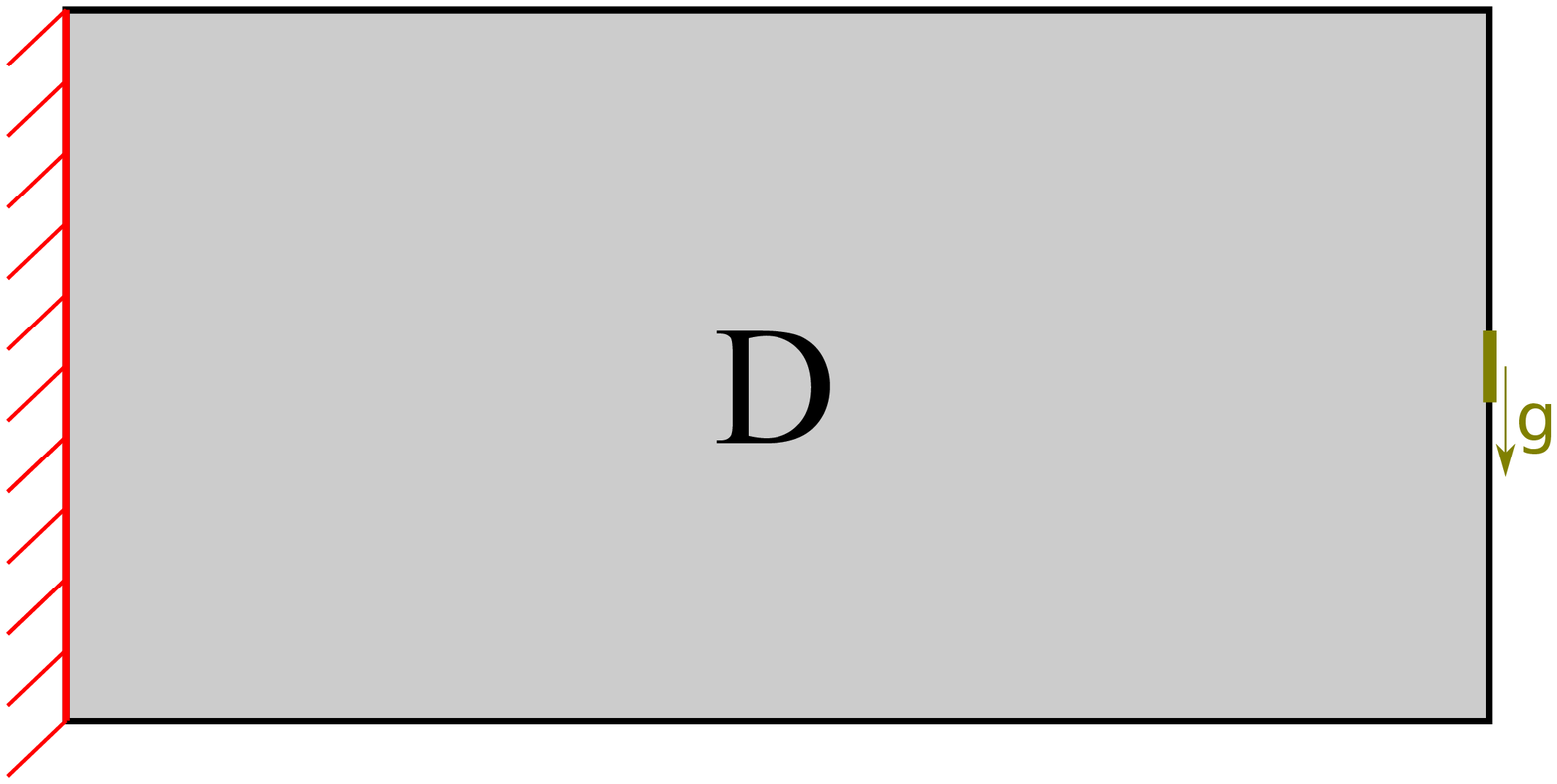} \end{minipage}&&
       \begin{minipage}{0.3\textwidth} \centering \includegraphics[width=\textwidth]{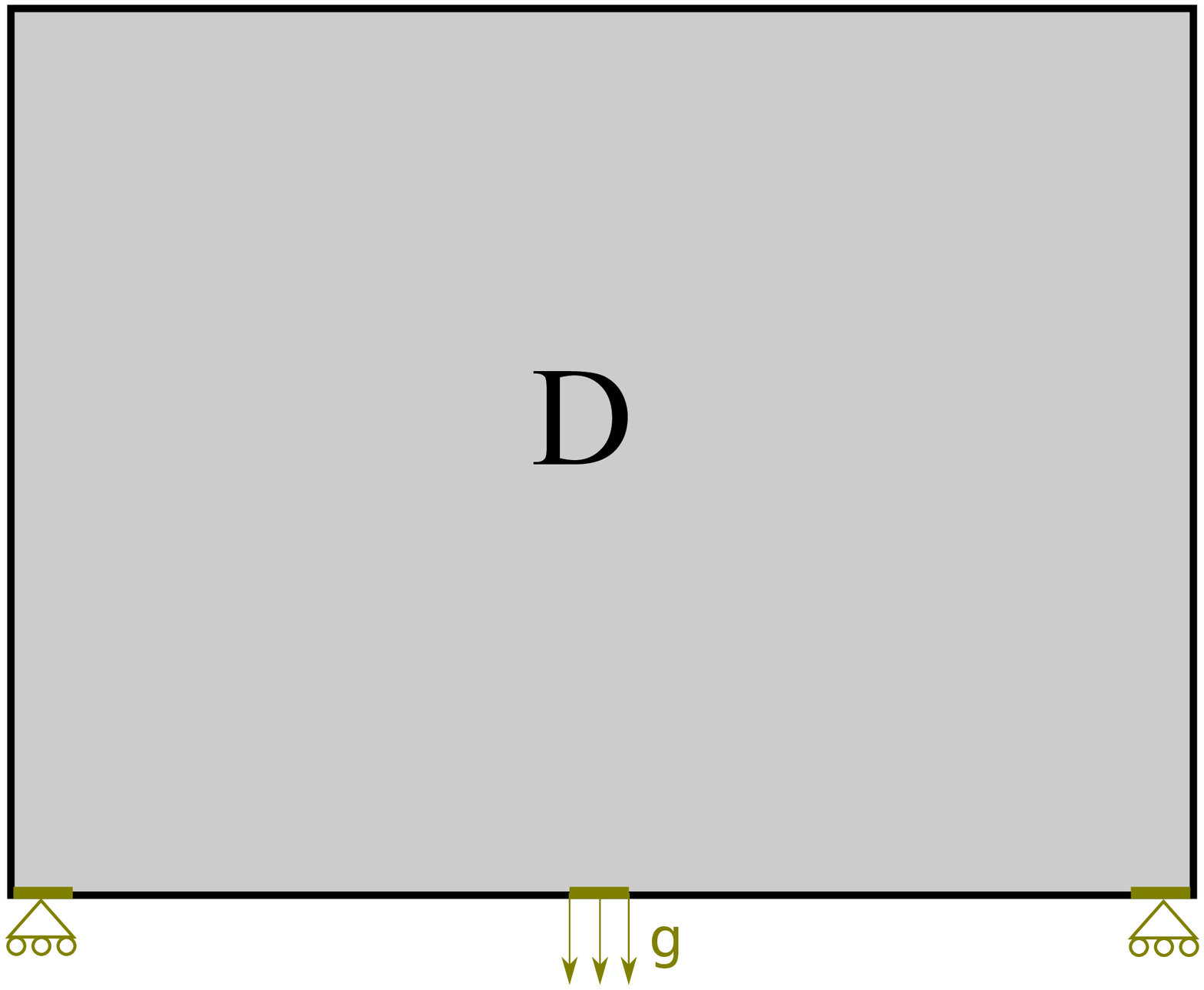} \end{minipage} &&
        \begin{minipage}{0.3\textwidth} \centering \includegraphics[width=\textwidth]{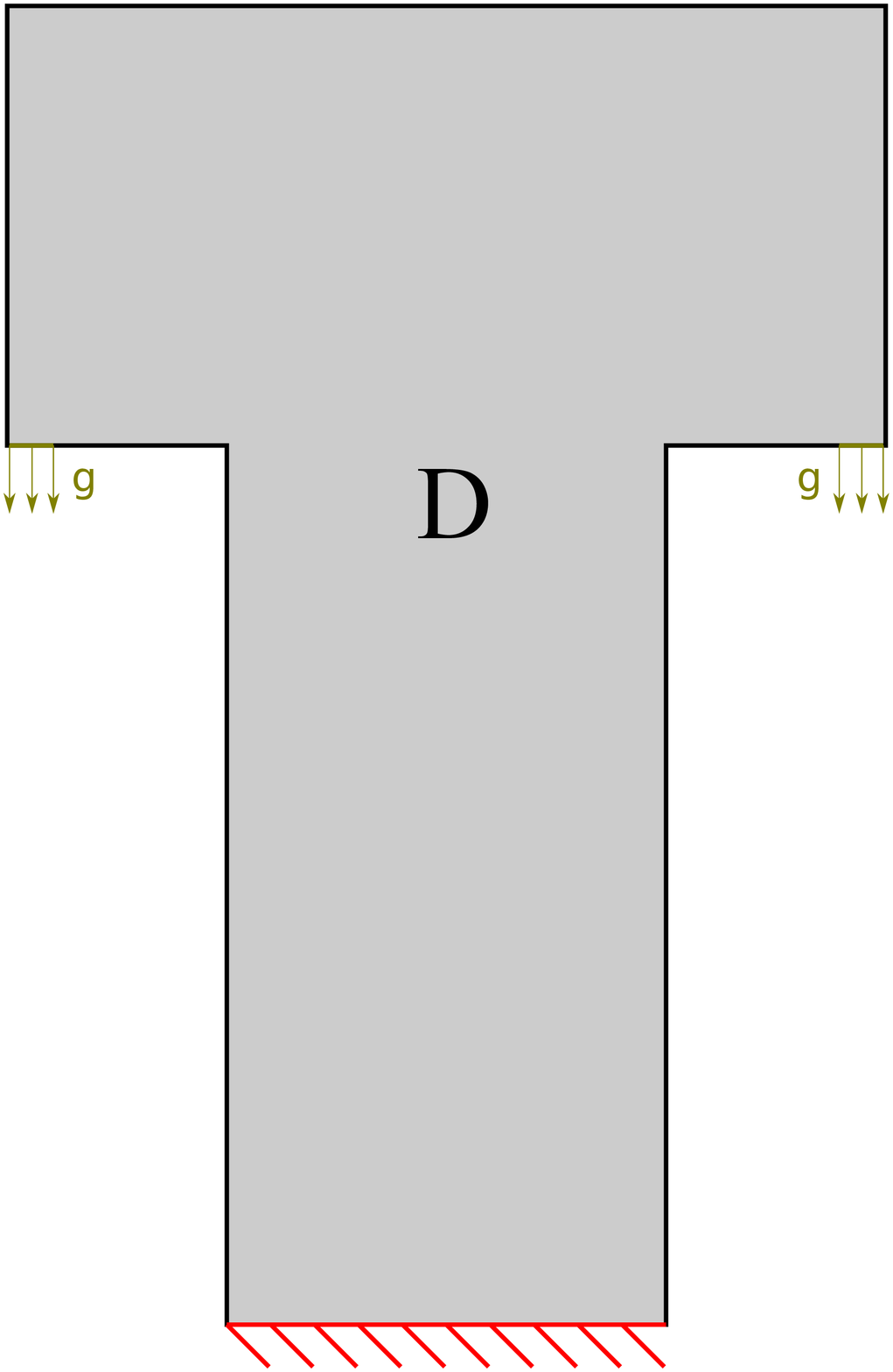} \end{minipage} \\
        (a) && (b) && (c)
    \end{tabular}
    \caption{Settings for three model problems from linear elasticity: (a) long cantilever, (b) bridge, (c) mast.}
    \label{fig_settings}
\end{figure}

\subsubsection{Long Cantilever}
As a first numerical example, we consider a long cantilever where $D = (-1,1) \times (0,1)$, $\Gamma_D = \{-1\} \times (0,1)$, $\Gamma_N = \{1\} \times (0.45,0.55)$ and $g = (0, -1)^\top$.
We used a mesh with 45824 triangles and 23218 vertices and we chose a maximum stepsize $\overline \kappa = 0.12$ in this example.

As an initial design we chose the whole computational domain $D$ to be occupied with strong material. Figure \ref{fig_evo_canti} shows the evolution of the design in the course of Algorithm \ref{algo_TDLSMmat}. 
The objective value $\mathcal J$ \eqref{eq_functionalComp} is reduced from $4.42779$ to $1.7177$ and the angle (in an $L^2(D, \mathbb R^{M-1})$ sense) between the level set function and the generalized topological derivative $G$ defined in \eqref{eq_genTD_mMat} is reduced from $123.5$ degrees to approximately $0.4$ degrees. Figure \ref{fig_graphs_canti} shows the evolution of the objective value, of the compliance value, of the angle and the volumes of the strong and weak materials in the course of the optimization iterations.

\begin{figure}
	\begin{tabular}{cc}
		\includegraphics[width=.5\textwidth]{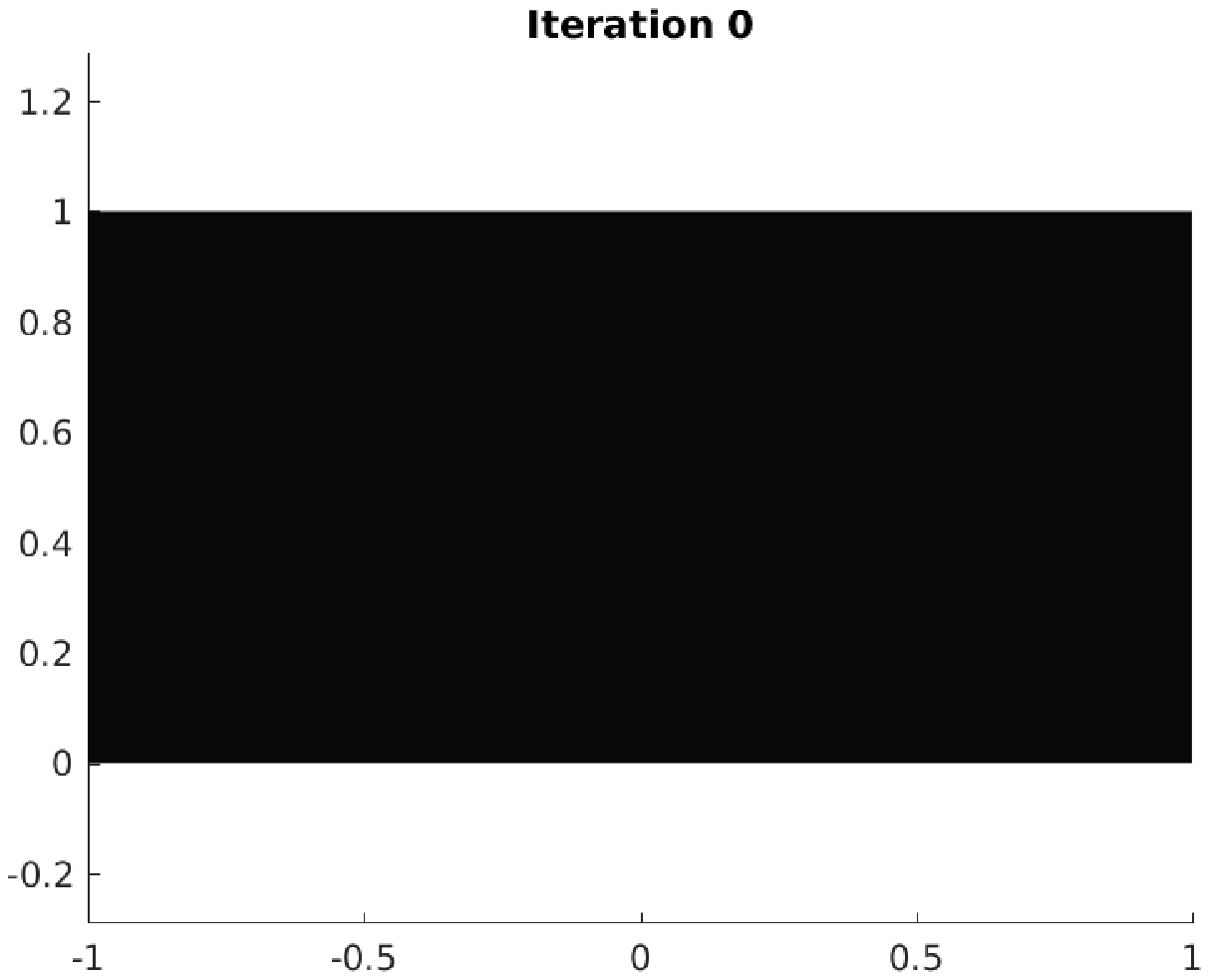}&\includegraphics[width=.5\textwidth]{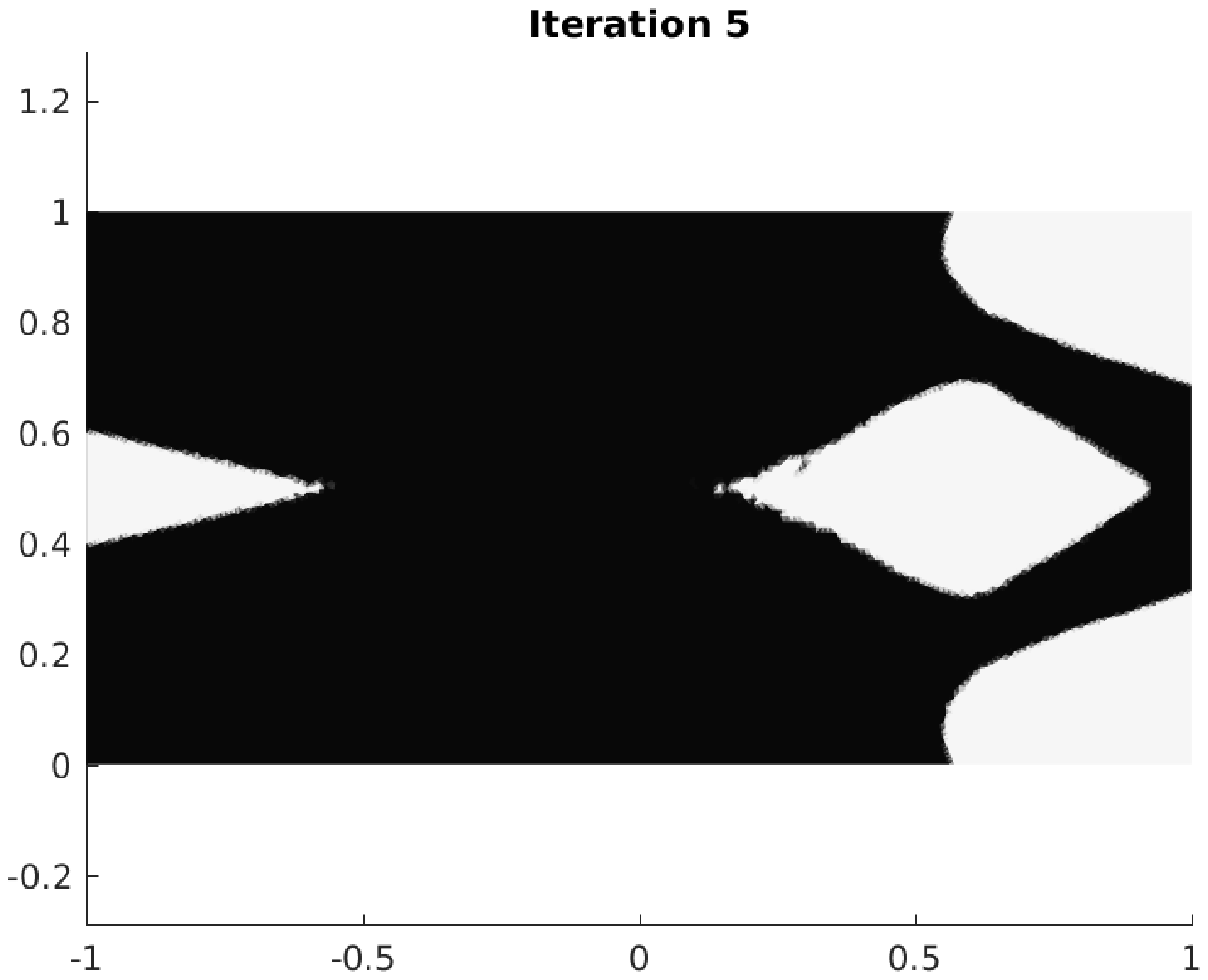} \\
		\includegraphics[width=.5\textwidth]{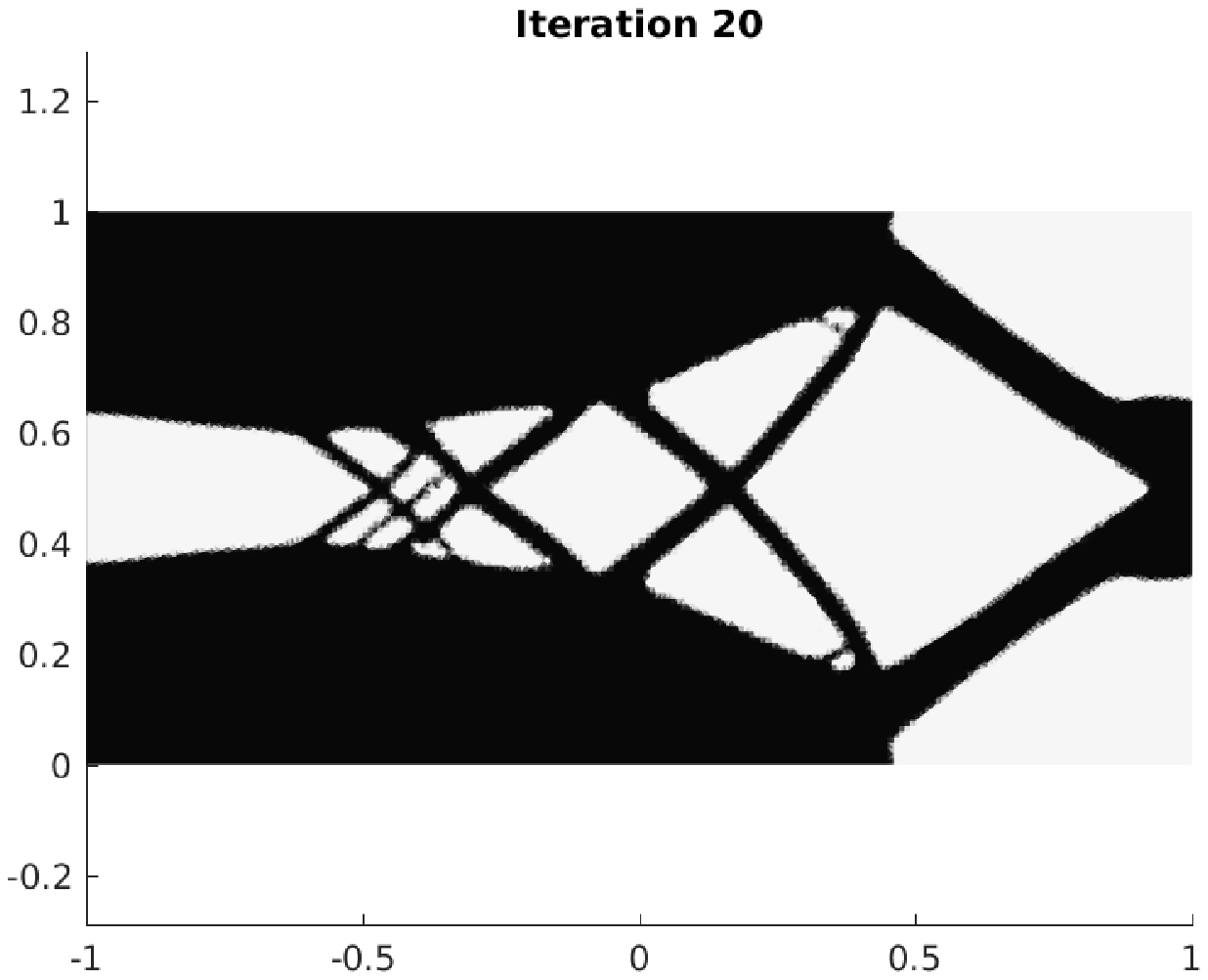}&\includegraphics[width=.5\textwidth]{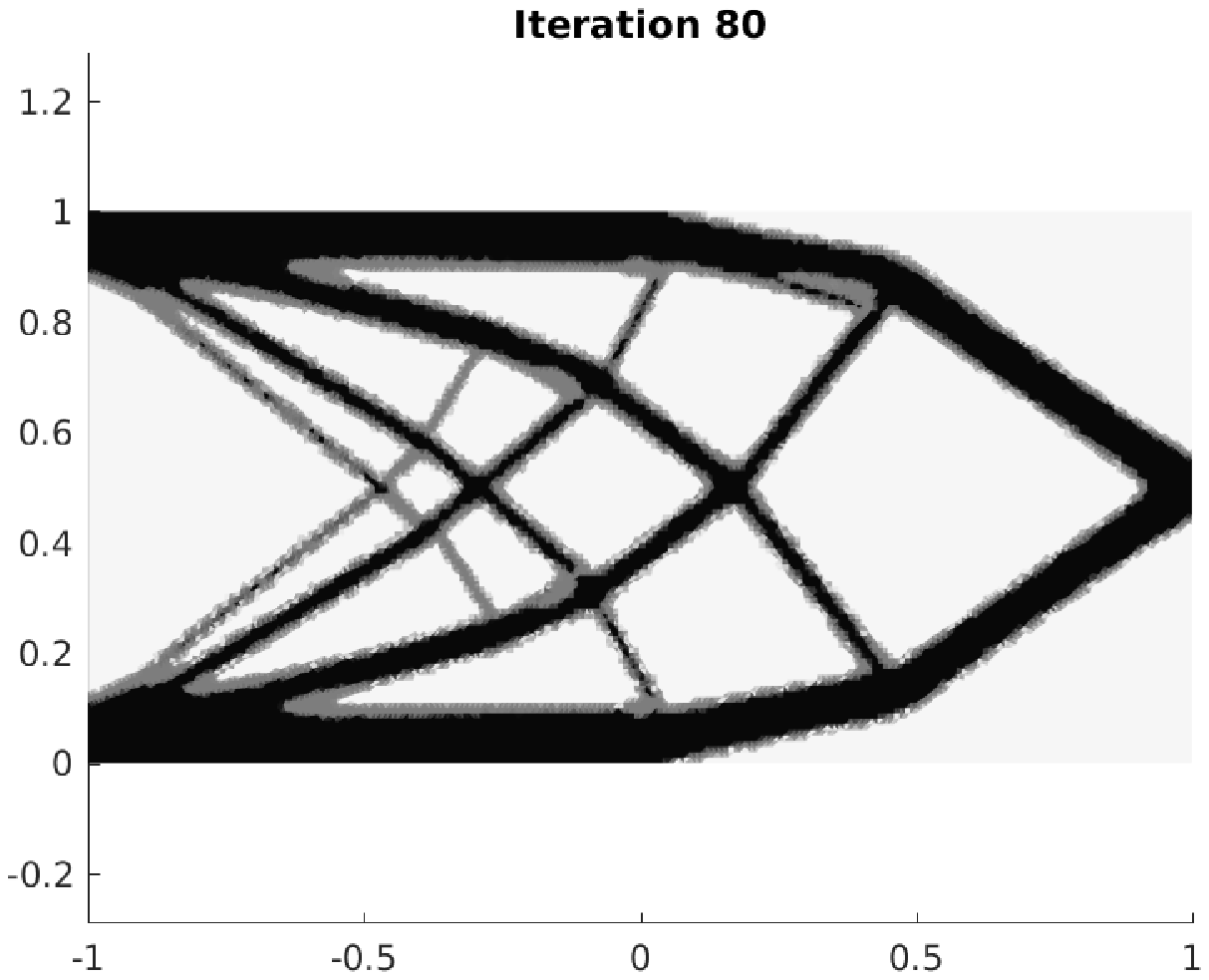} \\
		\includegraphics[width=.5\textwidth]{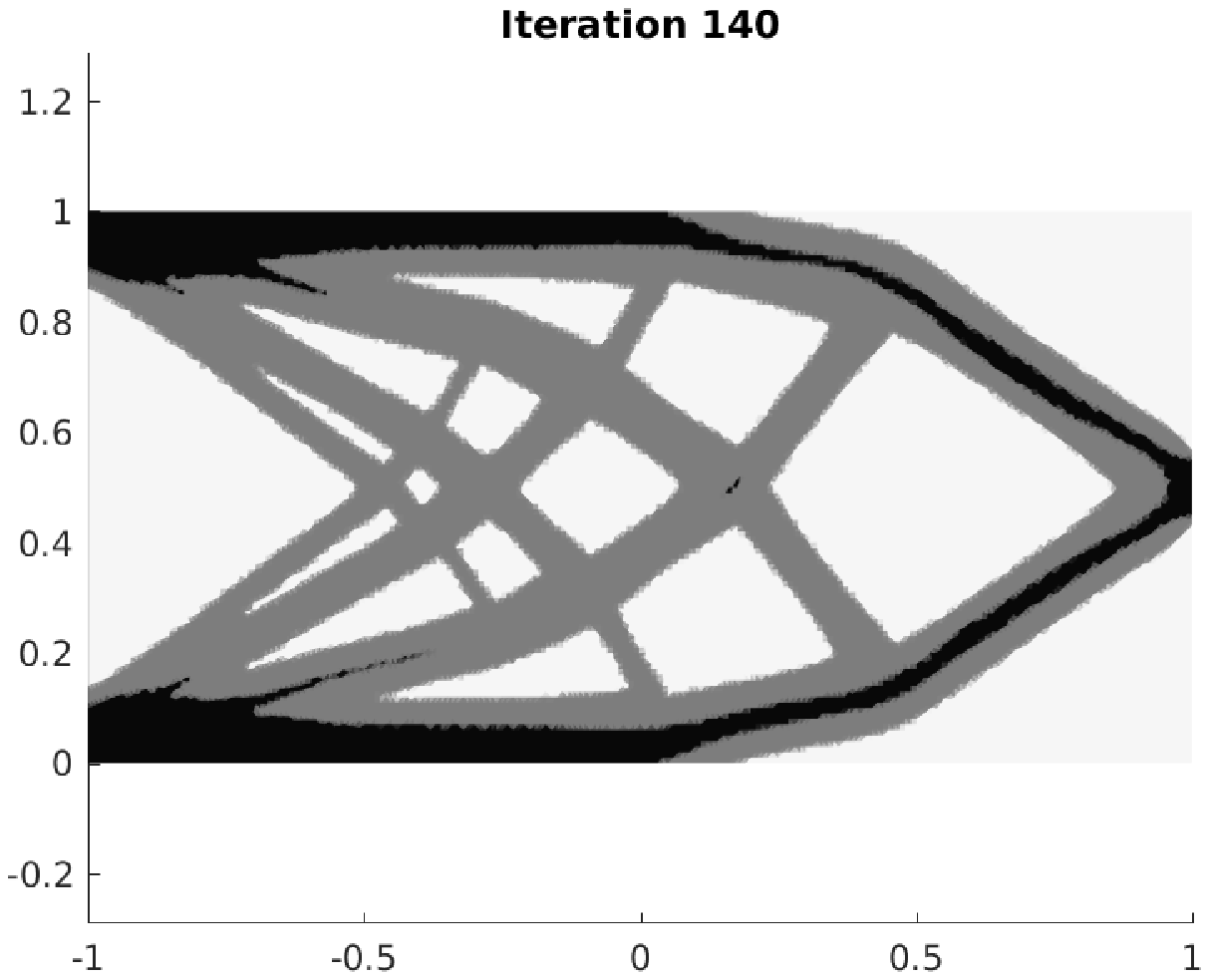}&\includegraphics[width=.5\textwidth]{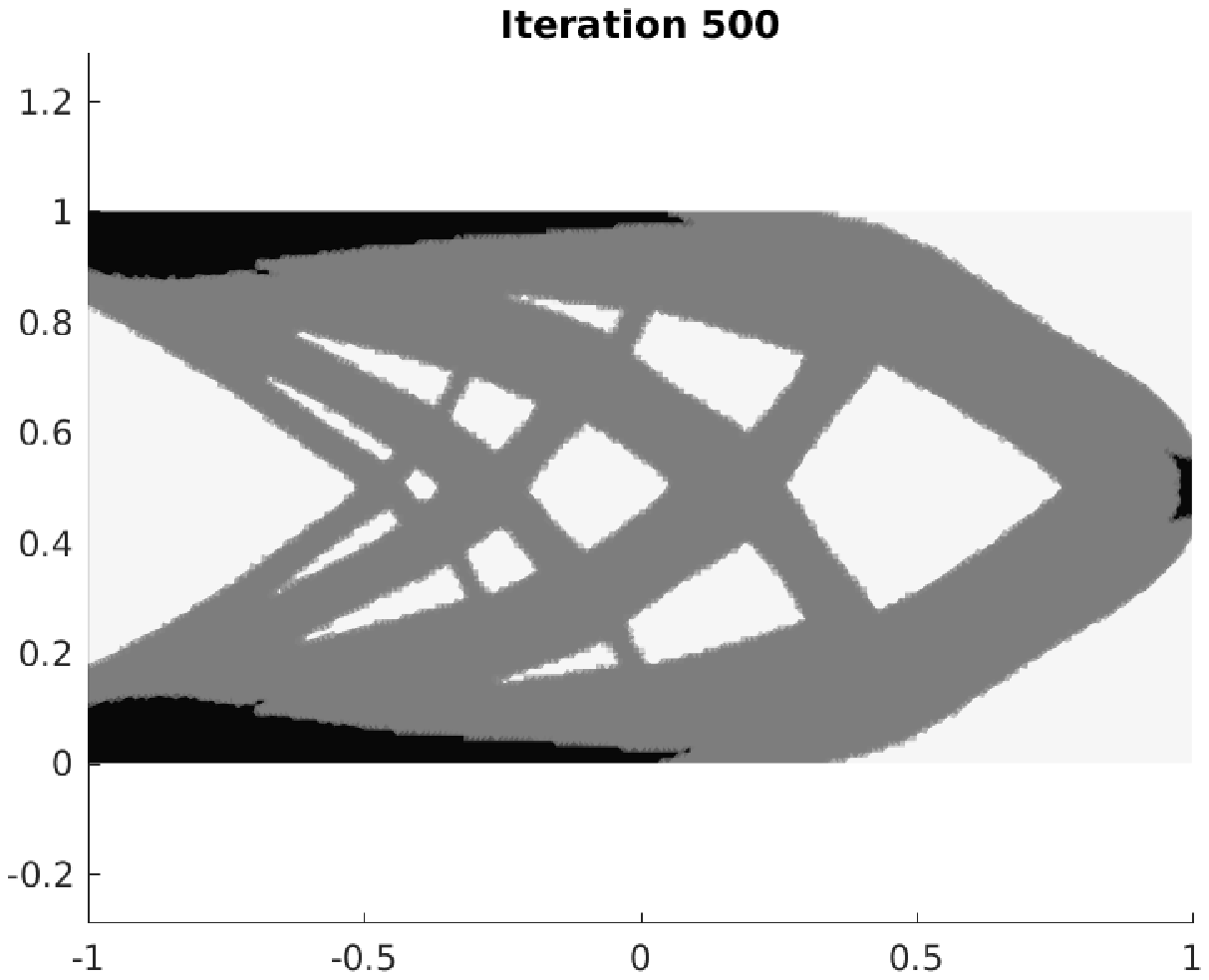} 
	\end{tabular}
	\caption{Evolution of design in the course of multi-material topology optimization by Algorithm \ref{algo_TDLSMmat} for long cantilever example. Black color corresponds to the strong material ($\Omega_1$) and gray color corresponds to the weaker material ($\Omega_2$).}
	\label{fig_evo_canti}
\end{figure}

\begin{figure}
    \hspace{-10mm}
	\begin{tabular}{cccc}
		\includegraphics[width=.25\textwidth]{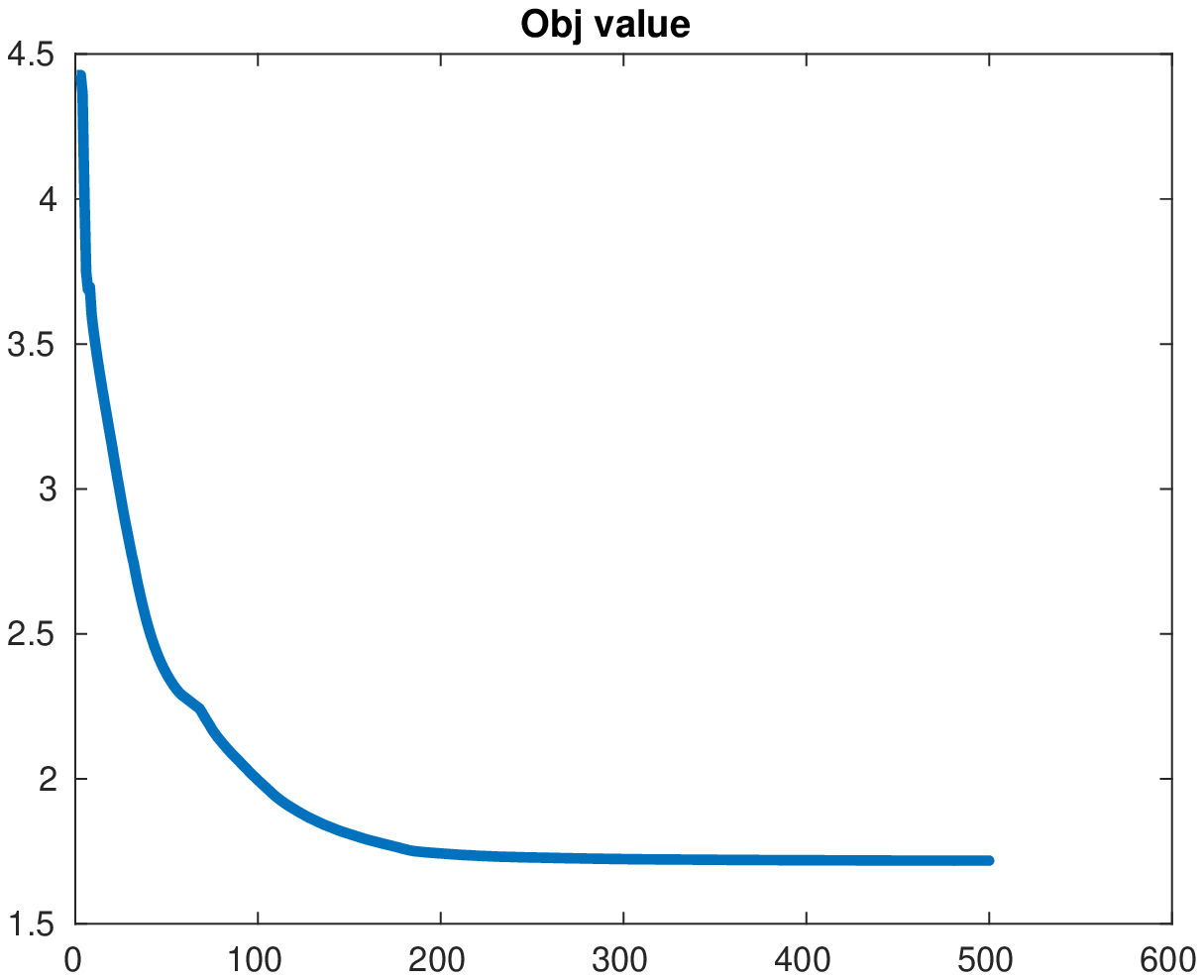}&\includegraphics[width=.25\textwidth]{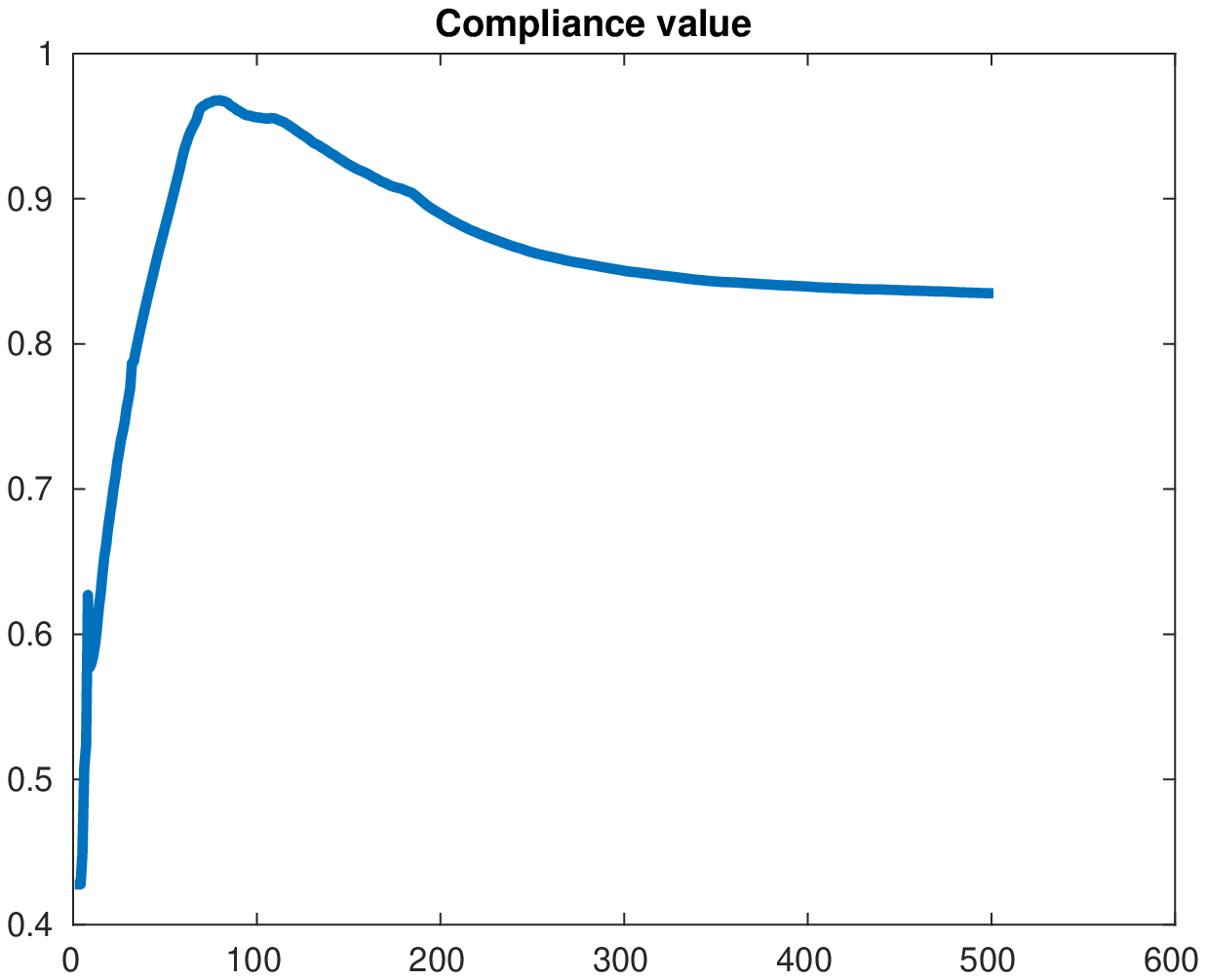} &
		\includegraphics[width=.25\textwidth]{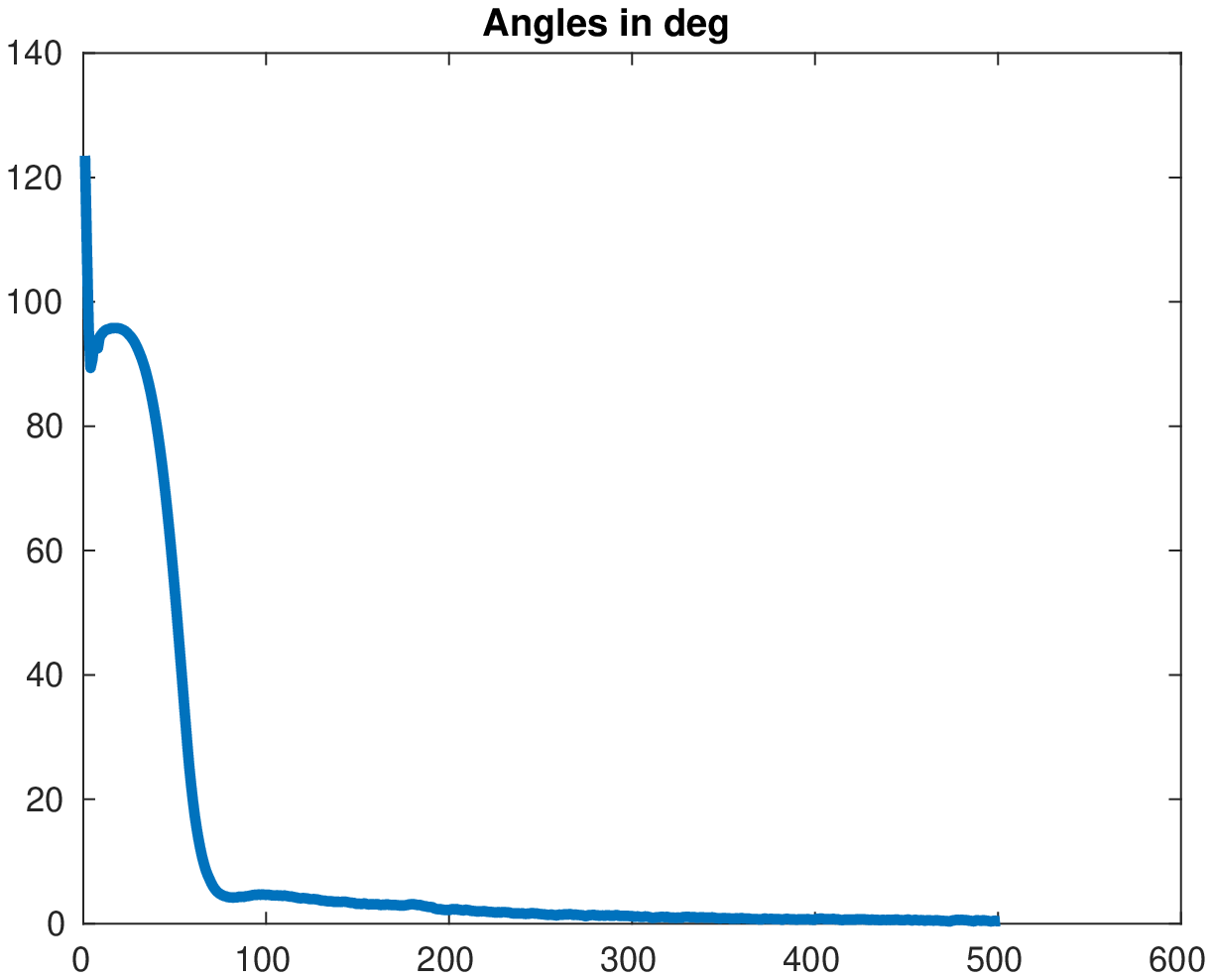}&\includegraphics[width=.25\textwidth]{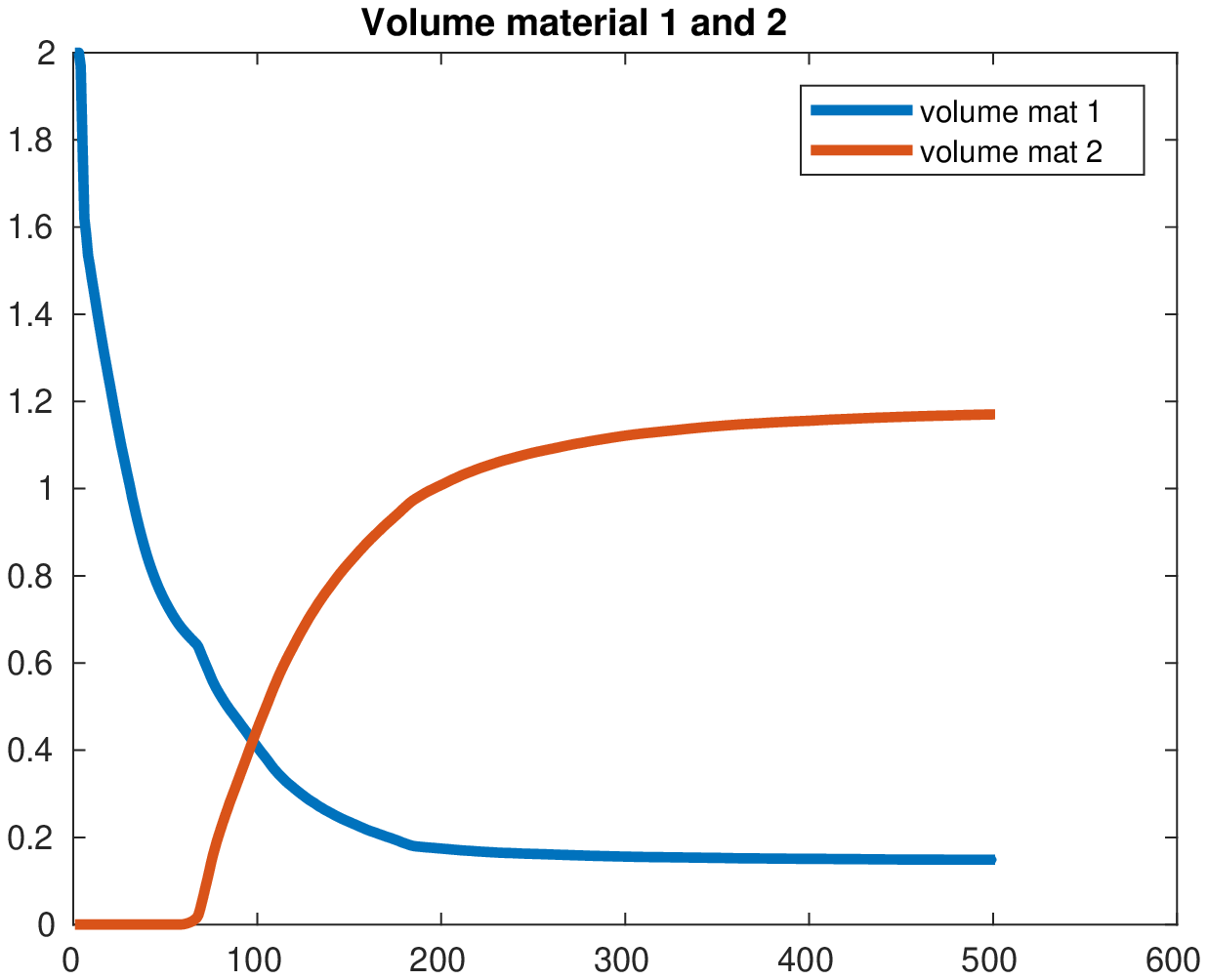} \\
		(a) & (b) & (c) & (d)
	\end{tabular}
	\caption{Evolution of different quantities in the course of Algorithm \ref{algo_TDLSMmat} for optimization of long cantilever. (a) Objective value $\mathcal J$ \eqref{eq_functionalComp}. (b) Compliance value $\mathcal C$. (c) Angle $\theta$ between level set function $\psi$ and generalized topological derivative $G$ in $L^2(D,\mathbb R^{M-1})$ sense. (d) Volumes of strong and weak materials.}
	\label{fig_graphs_canti}
\end{figure}

\subsubsection{Bridge}
In this example, we search for the optimal distribution of three materials (strong, weak and void) within the design region $D = (-1,1)\times (0, 1.5)$ as depicted in Figure \ref{fig_settings}(b). The structure is subject to a load $g = (0,-1)^\top$ acting on the bottom in the center, $\Gamma_N = (0.95, 1.05)\times \{0\}$. Moreover, the second component of the displacement vector $u$ is bound to vanish on the left and right bottom regions $(-1, -0.9) \times \{0\}$ and $(0.9, 1) \times \{0\}$. We used a grid consisting of 55144 triangles and 27878 vertices and chose the maximum allowed step size as $\overline \kappa = 0.2$. The evolution of the design is shown in Figure \ref{fig_evo_bridge}. The objective value $\mathcal J$ \eqref{eq_functionalComp} is reduced from $6.04865$ to $0.646332$ and the angle from $126.8$ degrees to $4.467 \cdot 10^{-5}$ degrees, see Figure \ref{fig_graphs_bridge}.

\begin{figure}
	\begin{tabular}{cc}
		\includegraphics[width=.5\textwidth]{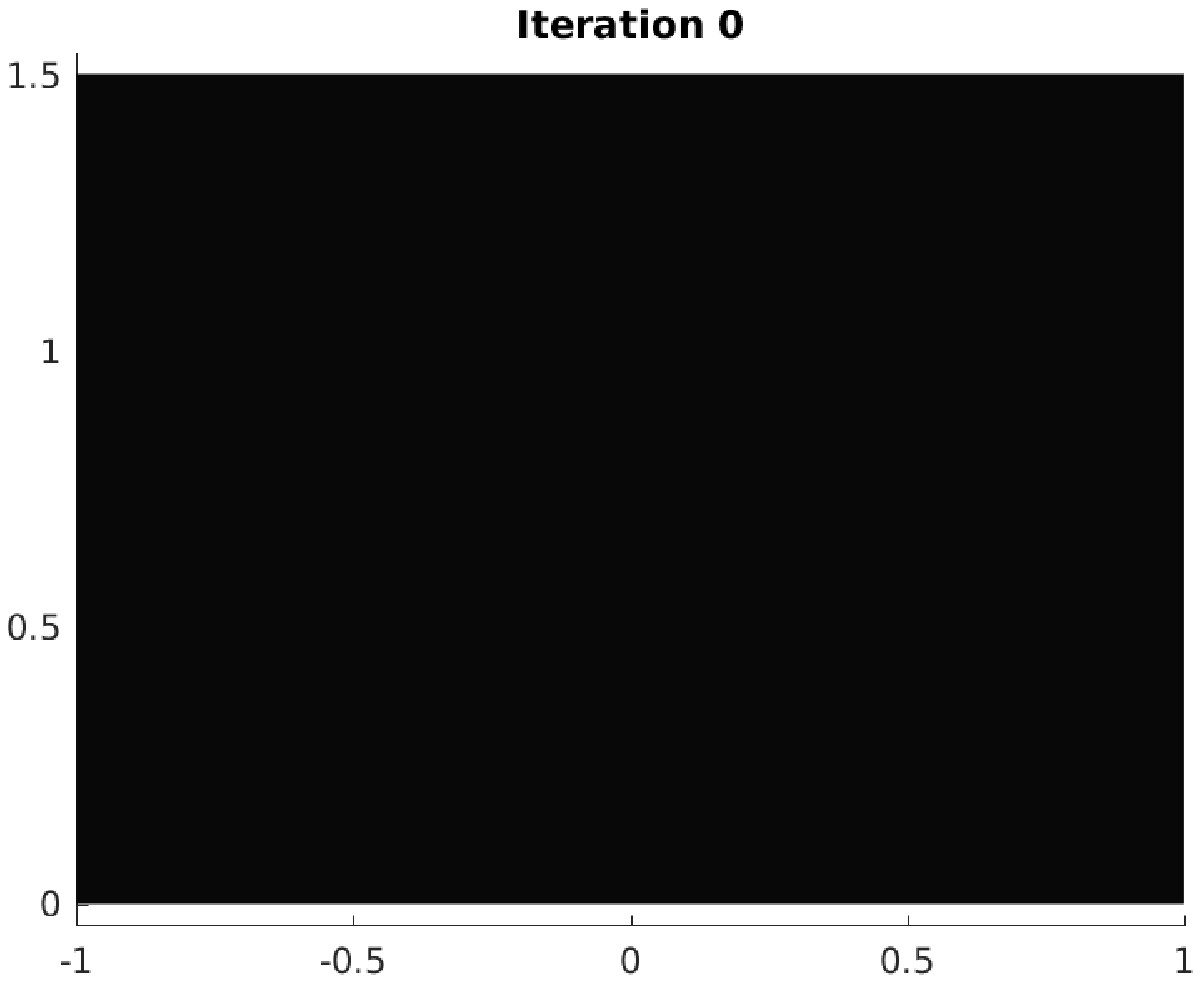}&\includegraphics[width=.5\textwidth]{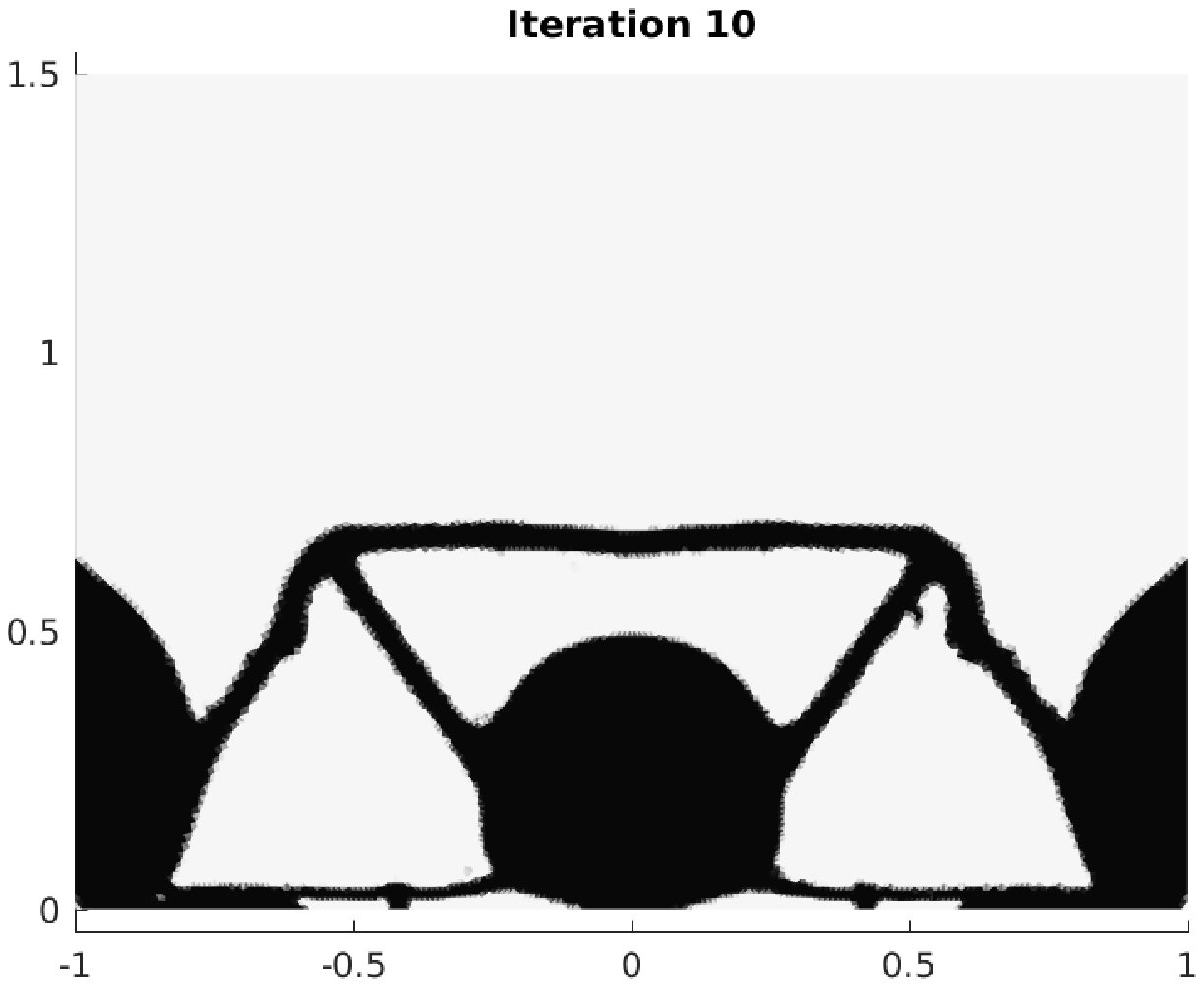} \\
		\includegraphics[width=.5\textwidth]{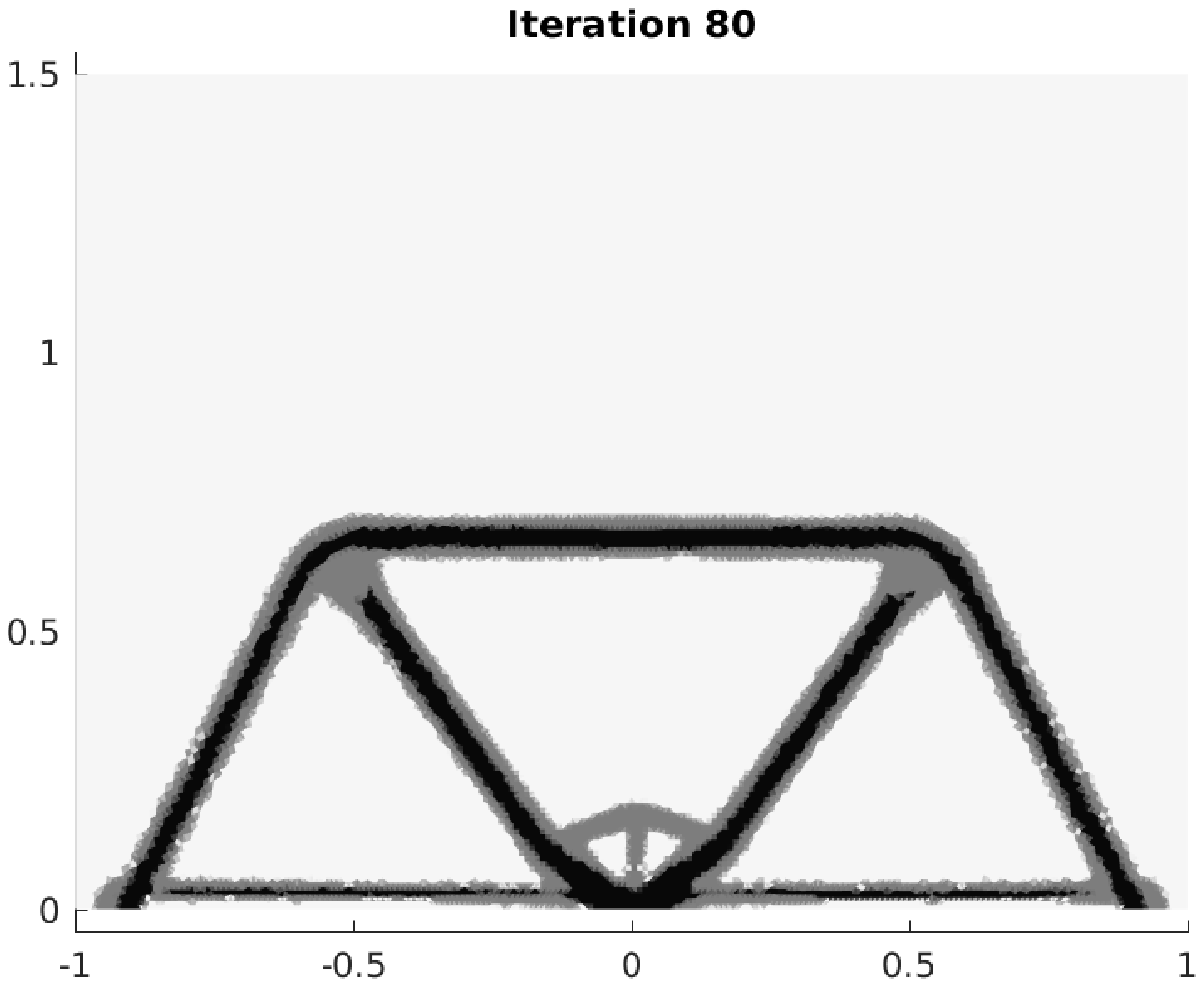}&\includegraphics[width=.5\textwidth]{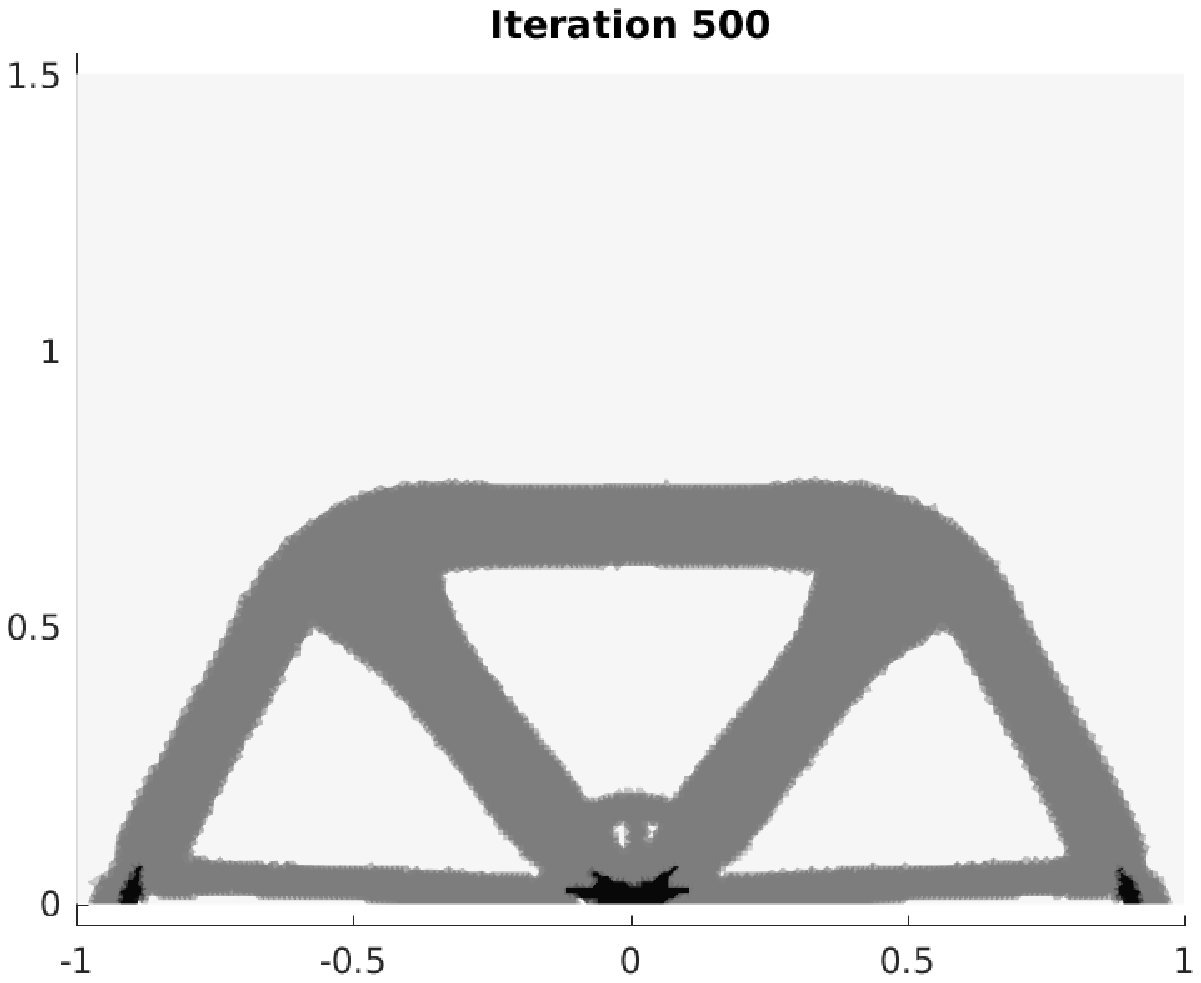} 
	\end{tabular}
	\caption{Evolution of design in the course of Algorithm \ref{algo_TDLSMmat} for bridge example. Black color corresponds to the strong material ($\Omega_1$) and gray color corresponds to the weaker material~($\Omega_2$).}
	\label{fig_evo_bridge}
\end{figure}

\begin{figure}
    \begin{center}
	\begin{tabular}{ccc}
		\includegraphics[width=.4\textwidth]{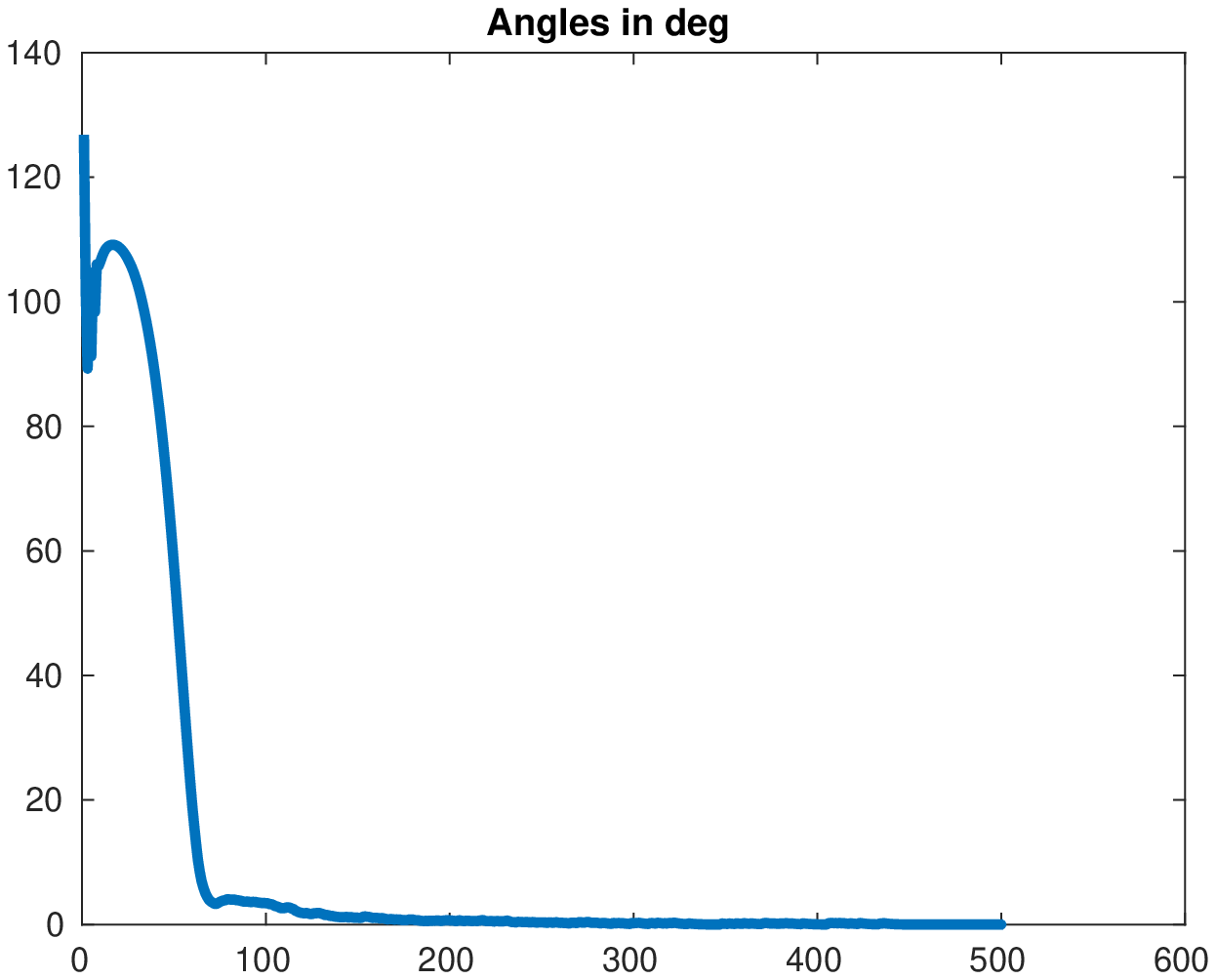}&&\includegraphics[width=.4\textwidth]{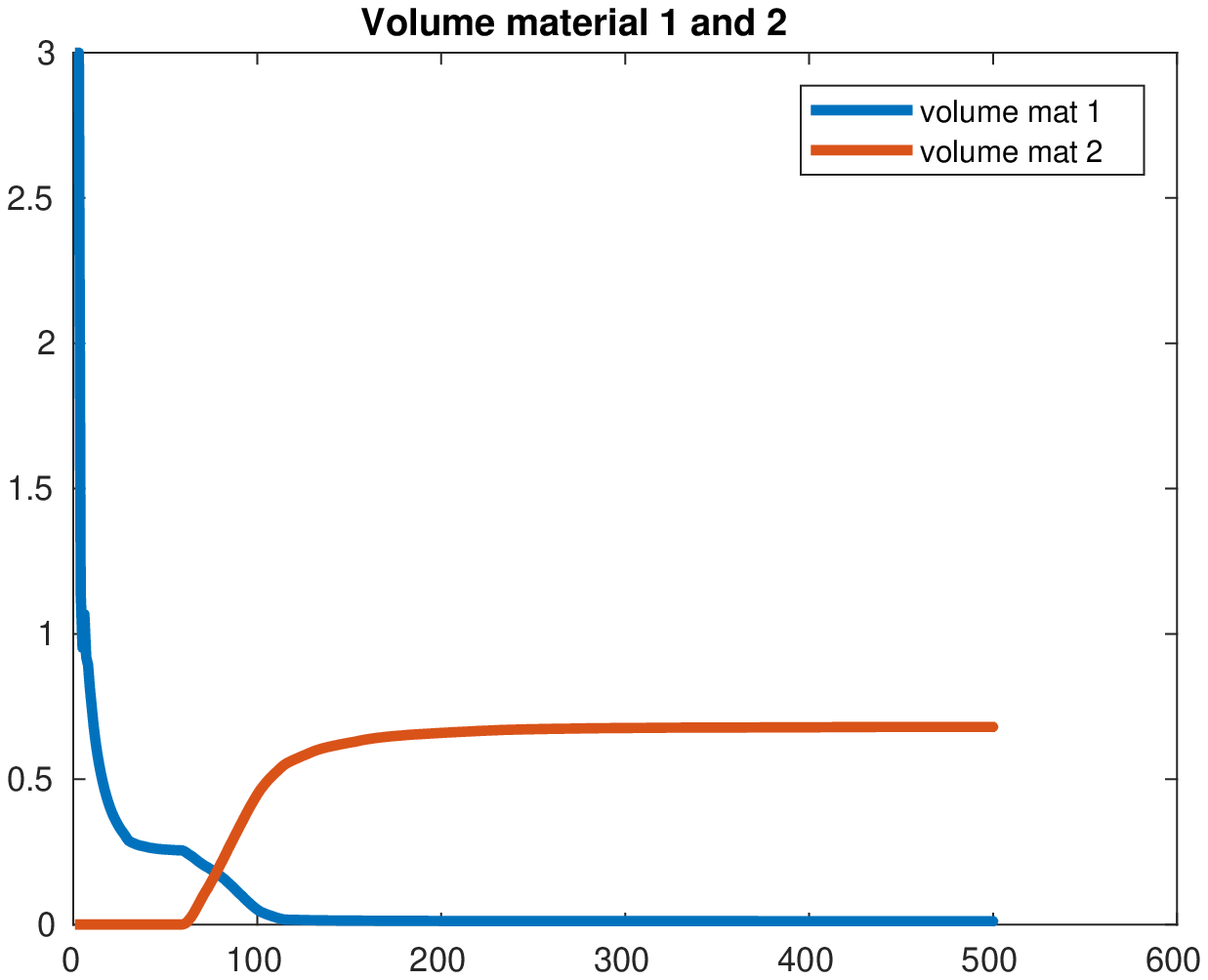} \\
		(a) & & (b)
	\end{tabular}
    \end{center}
	\caption{Evolution of angle and volumes in the course of Algorithm \ref{algo_TDLSMmat} applied to optimization of a bridge. (a) Angle between level set function and generalized topological derivative. (b) Volumes of strong material (material 1) and weak material (material 2).}
	\label{fig_graphs_bridge}
\end{figure}

\subsubsection{Mast}
In this example, we consider the computational domain shown in Figure \ref{fig_settings}(c), which we decomposed into a mesh of 50523 triangles and 25638 vertices. The structure is fixed on the bottom, $\Gamma_D = (-0.5, 0.5)\times \{0\}$, and is subject to a vertical load on the left and right parts, $\Gamma_N = (-1,-0.9) \times \{2\} \cup (0.9, 1) \times \{2\}$. We used a maximum step size $\overline \kappa = 0.1$. The evolution of the design in the course of Algorithm \ref{algo_TDLSMmat} is depicted in Figure \ref{fig_evo_mast}. The objective value is reduced from $8.22603$ to $1.80394$ and the angle $\theta$ from approximately $127$ degrees to approximately $0.46$ degrees, see also Figure \ref{fig_graphs_mast}.

\begin{figure}
	\begin{tabular}{cc}
		\includegraphics[width=.5\textwidth]{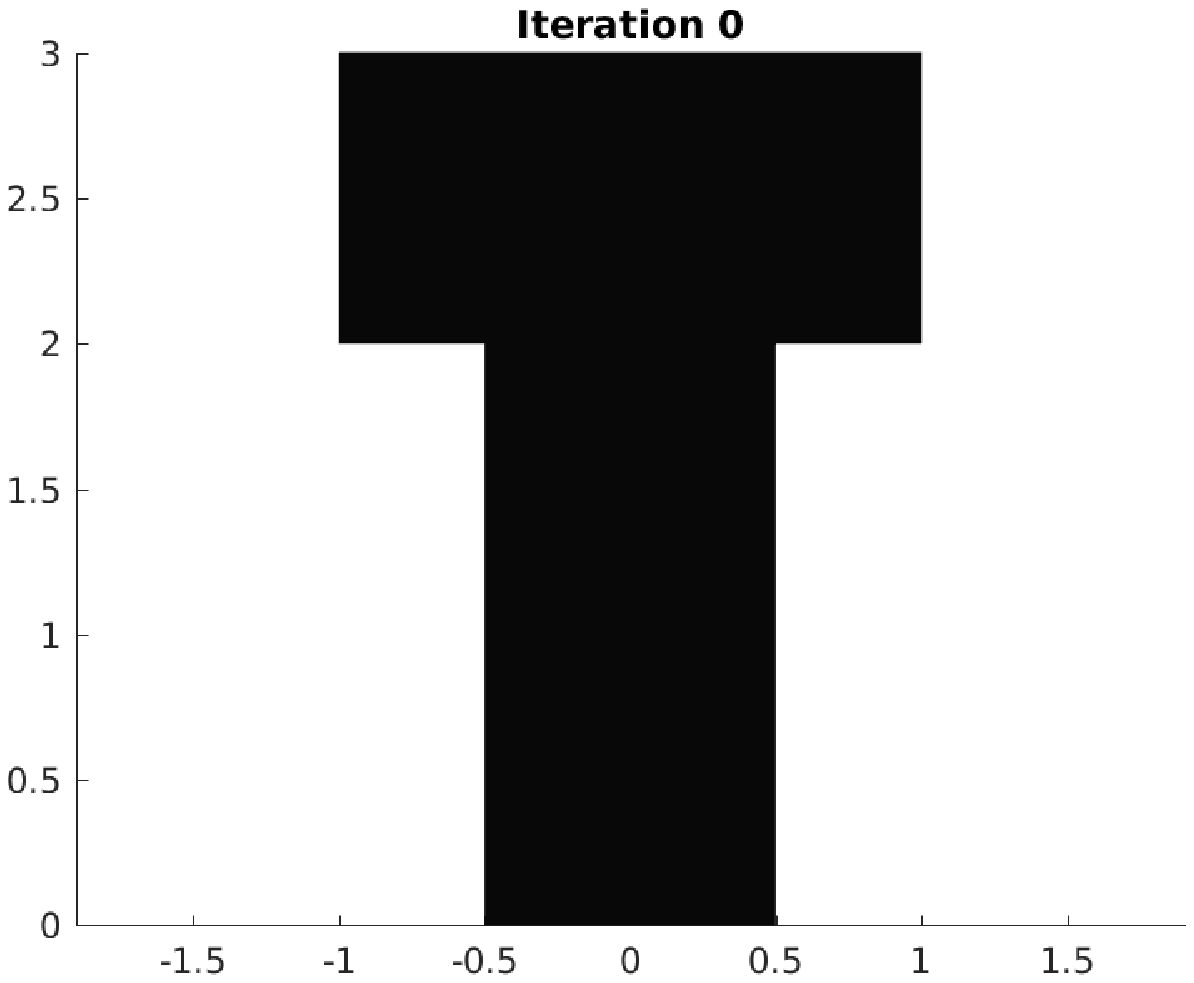}&\includegraphics[width=.5\textwidth]{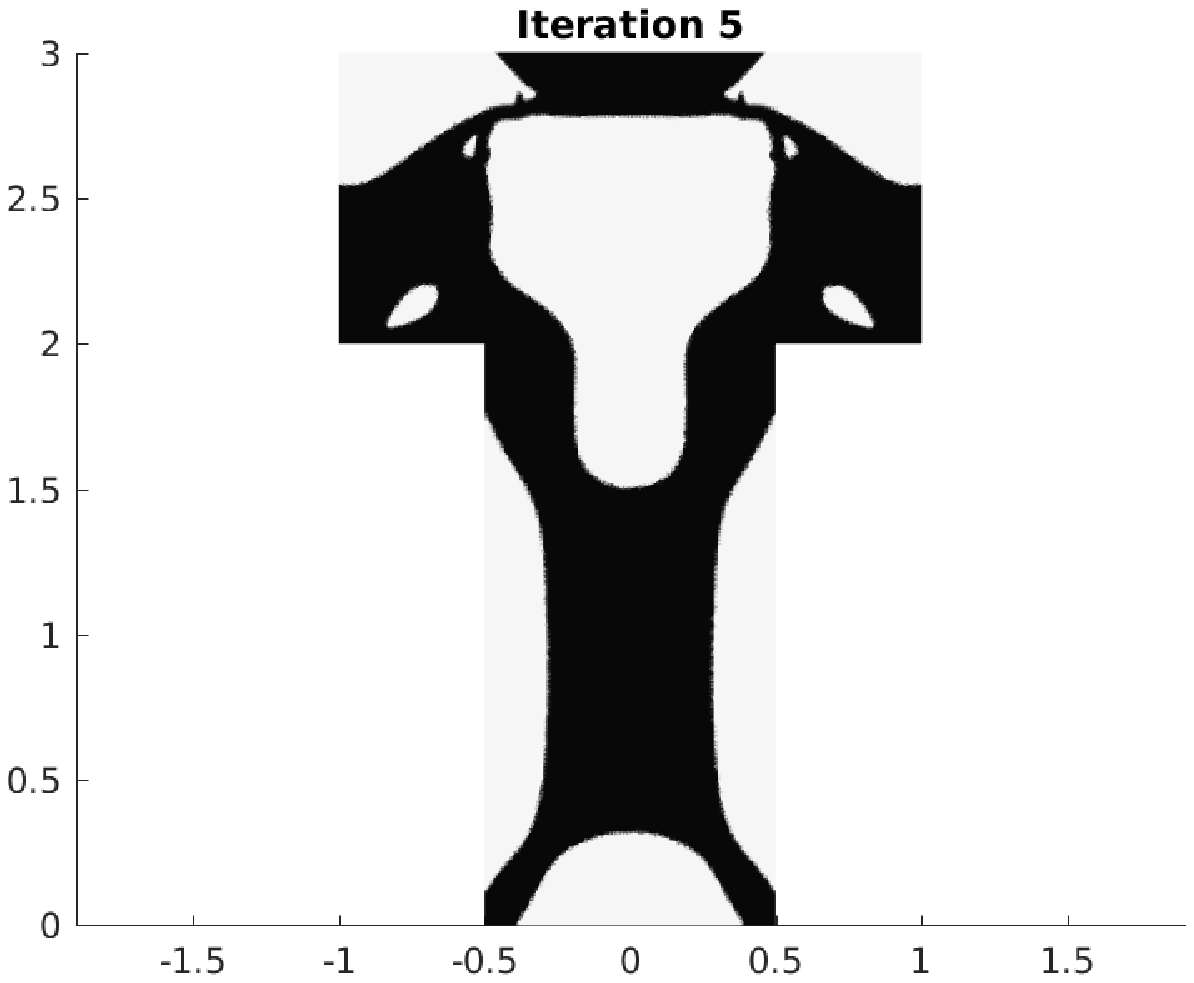} \\
		\includegraphics[width=.5\textwidth]{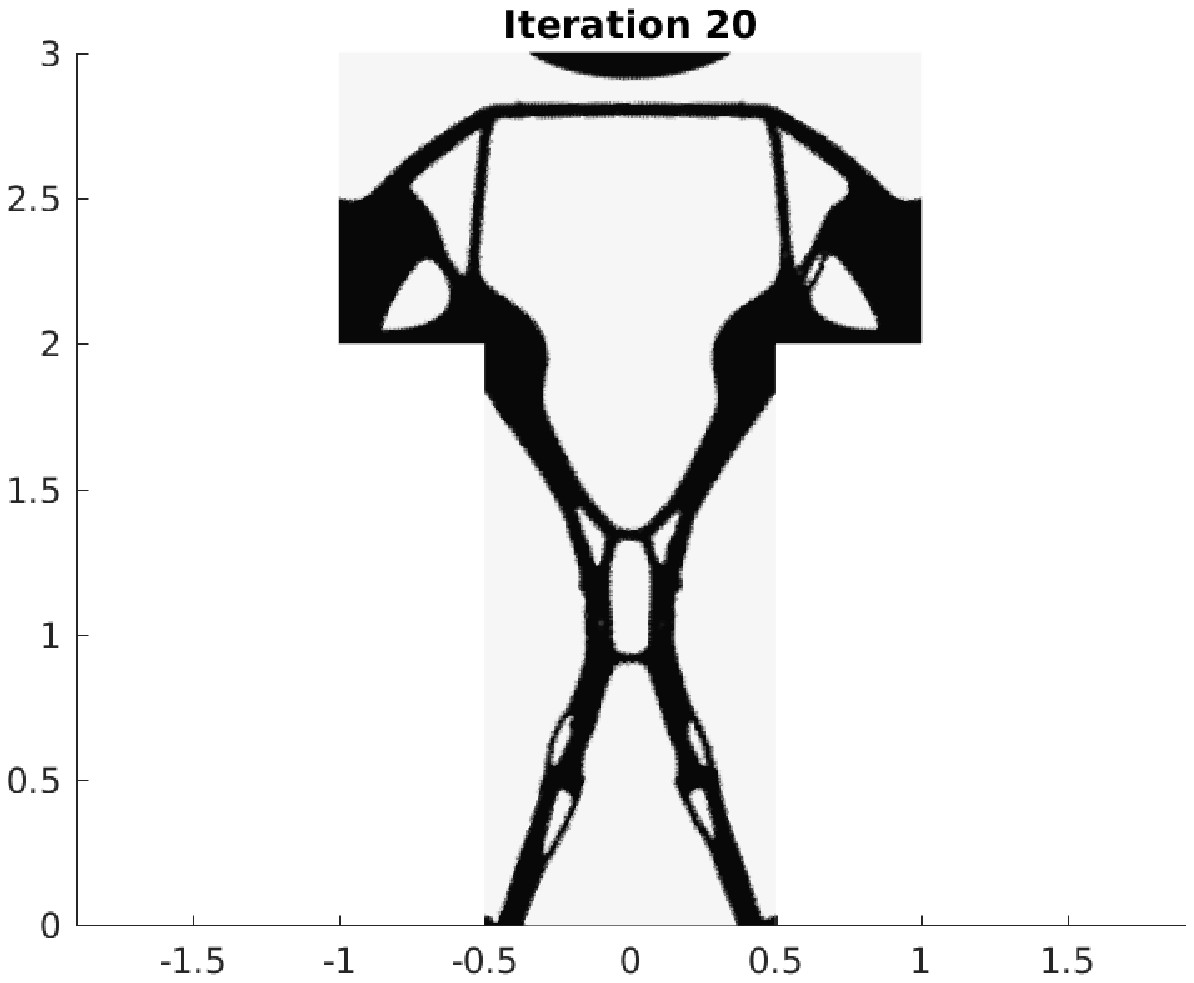}&\includegraphics[width=.5\textwidth]{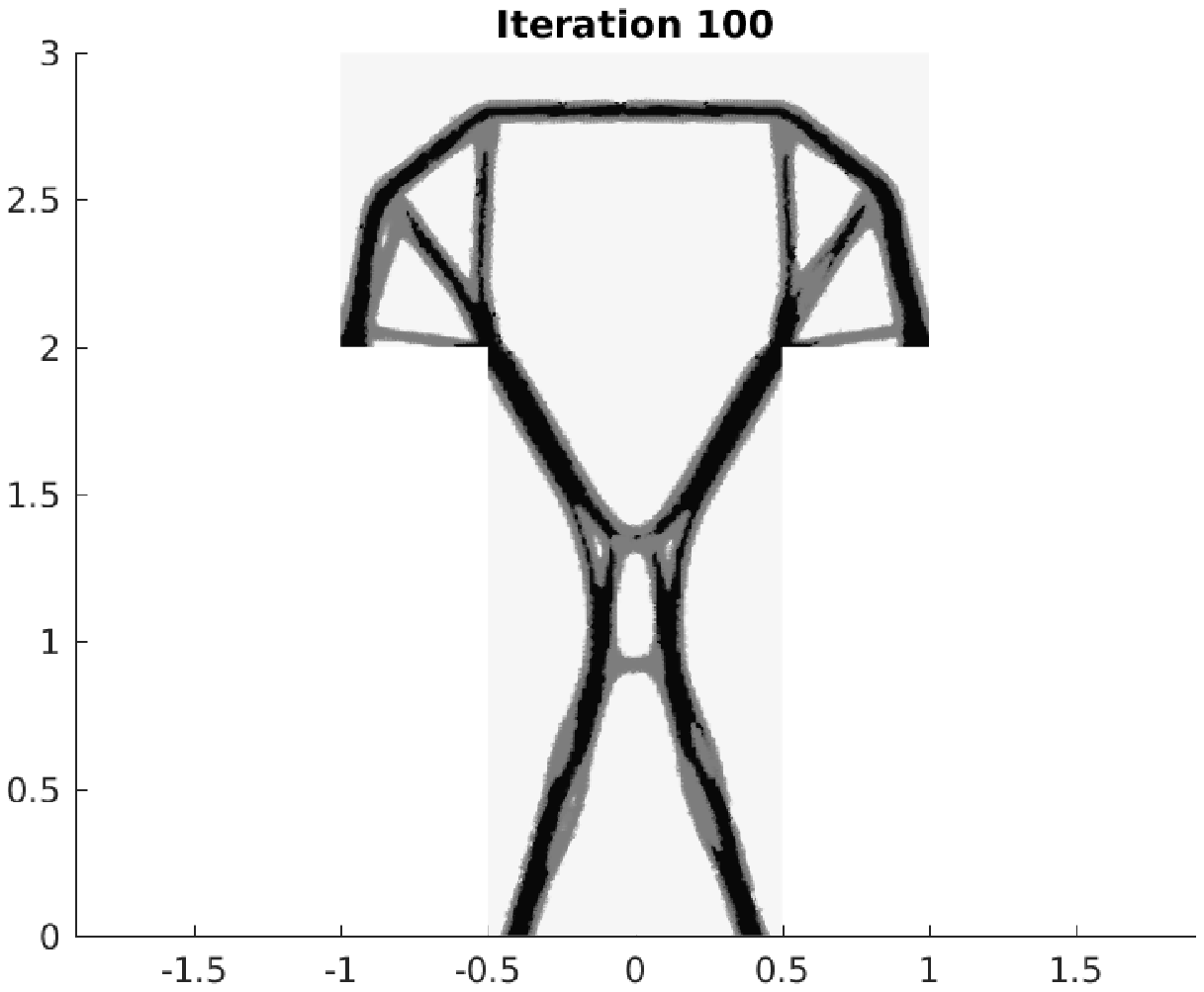} \\
		\includegraphics[width=.5\textwidth]{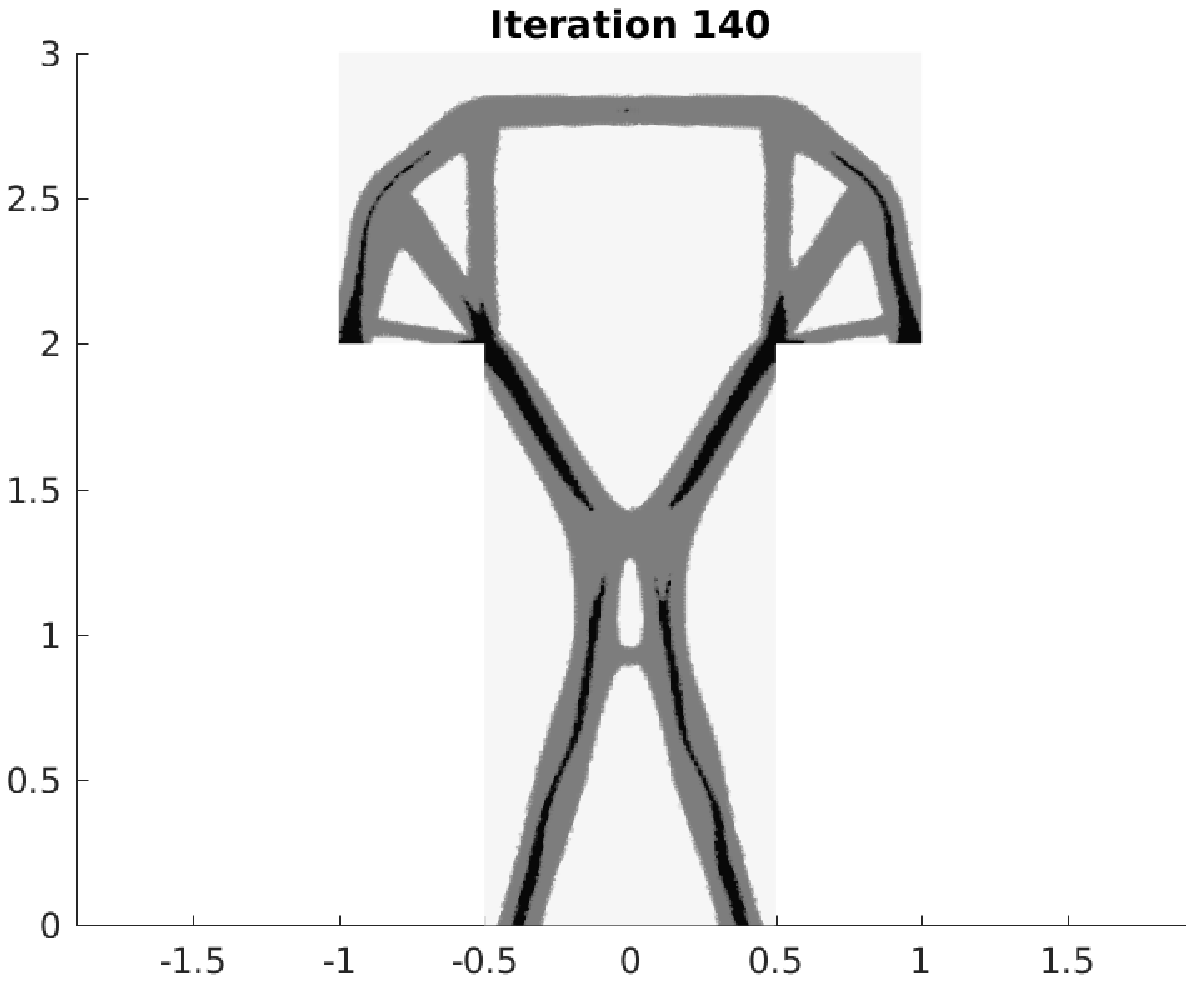}&\includegraphics[width=.5\textwidth]{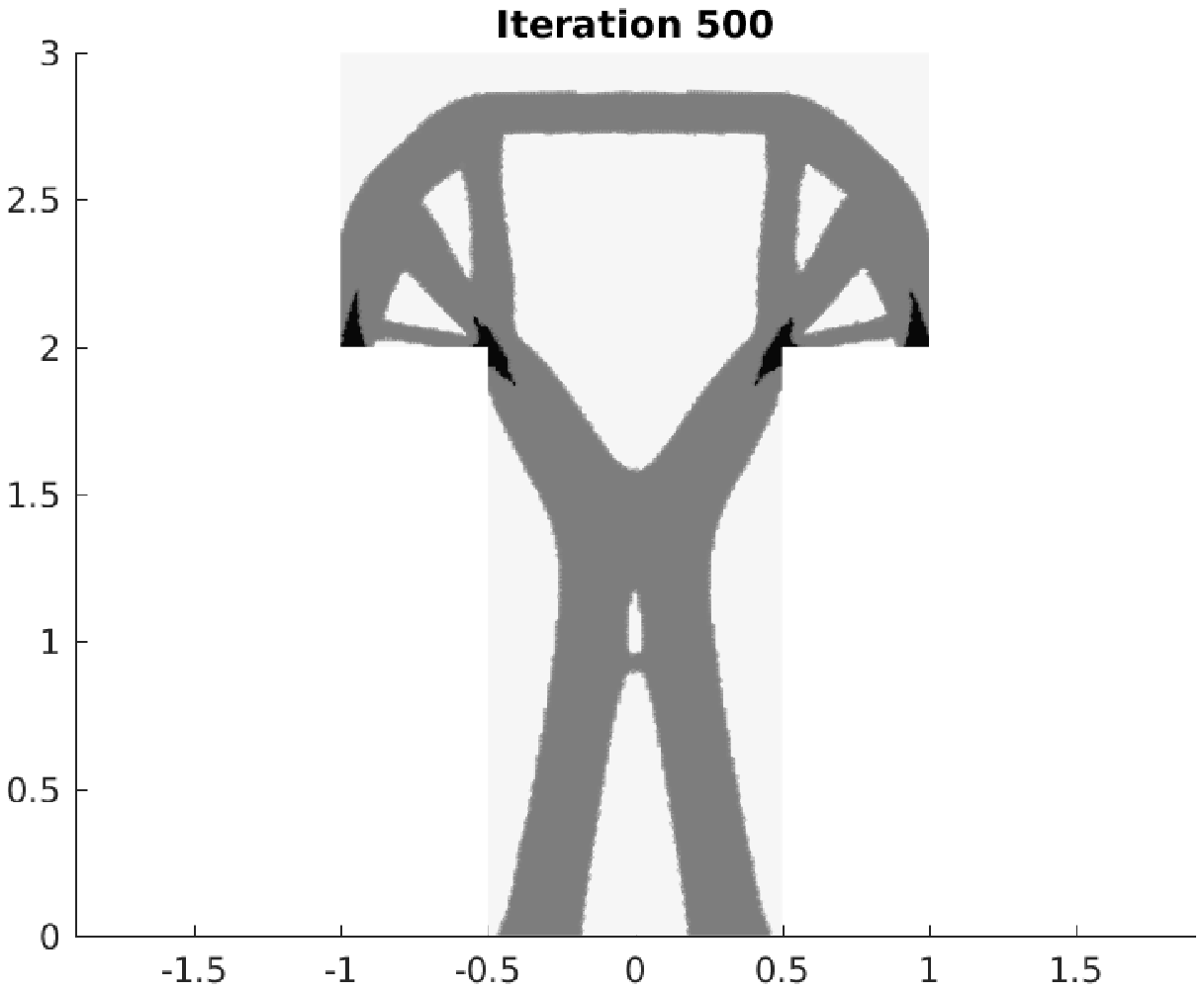} 
	\end{tabular}
	\caption{Evolution of design in the course of Algorithm \ref{algo_TDLSMmat} for mast example. Black color corresponds to the strong material ($\Omega_1$) and gray color corresponds to the weaker material~($\Omega_2$).}
	\label{fig_evo_mast}
\end{figure}

\begin{figure}
    \begin{center}
	\begin{tabular}{cc}
		\includegraphics[width=.4\textwidth]{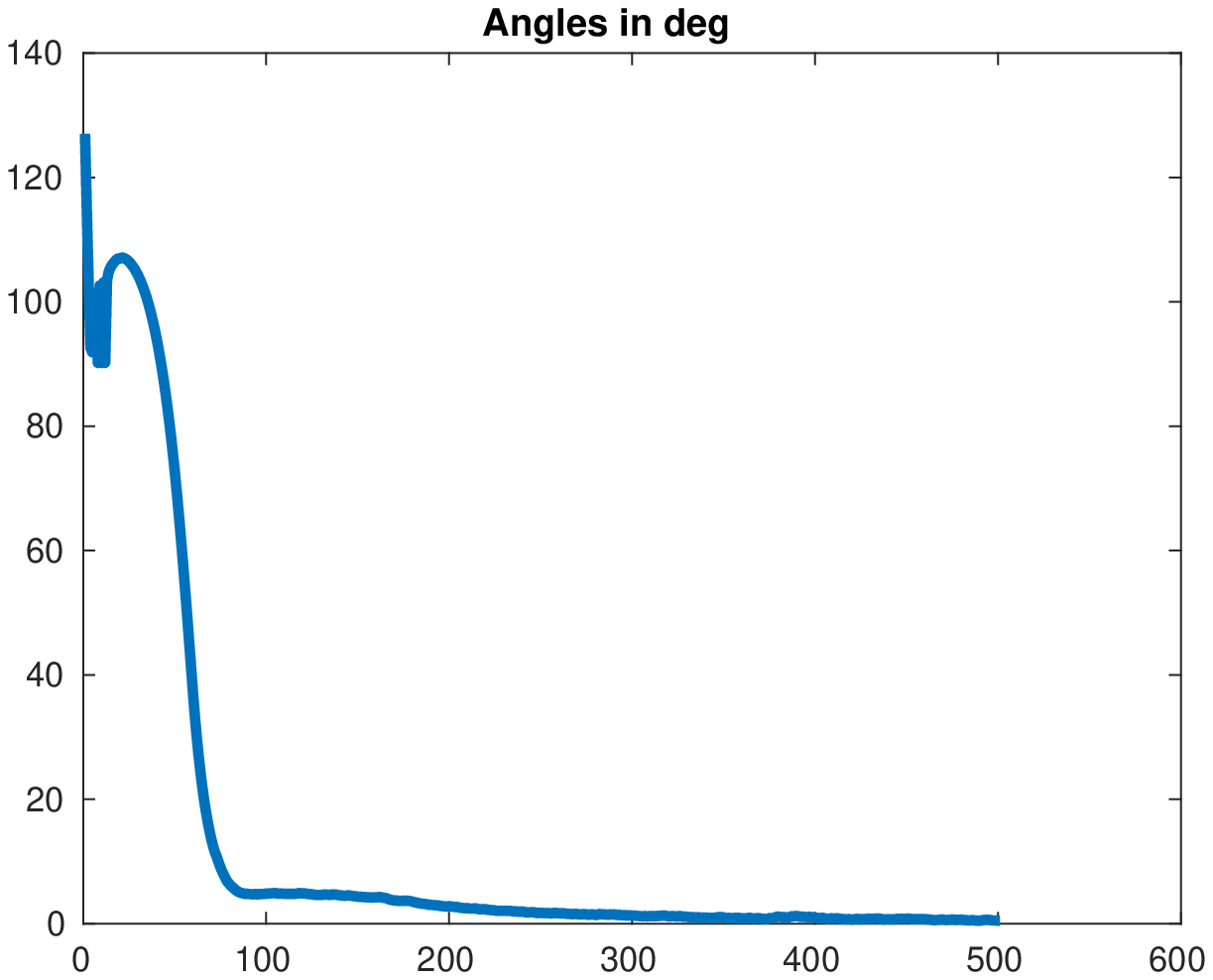}&\includegraphics[width=.4\textwidth]{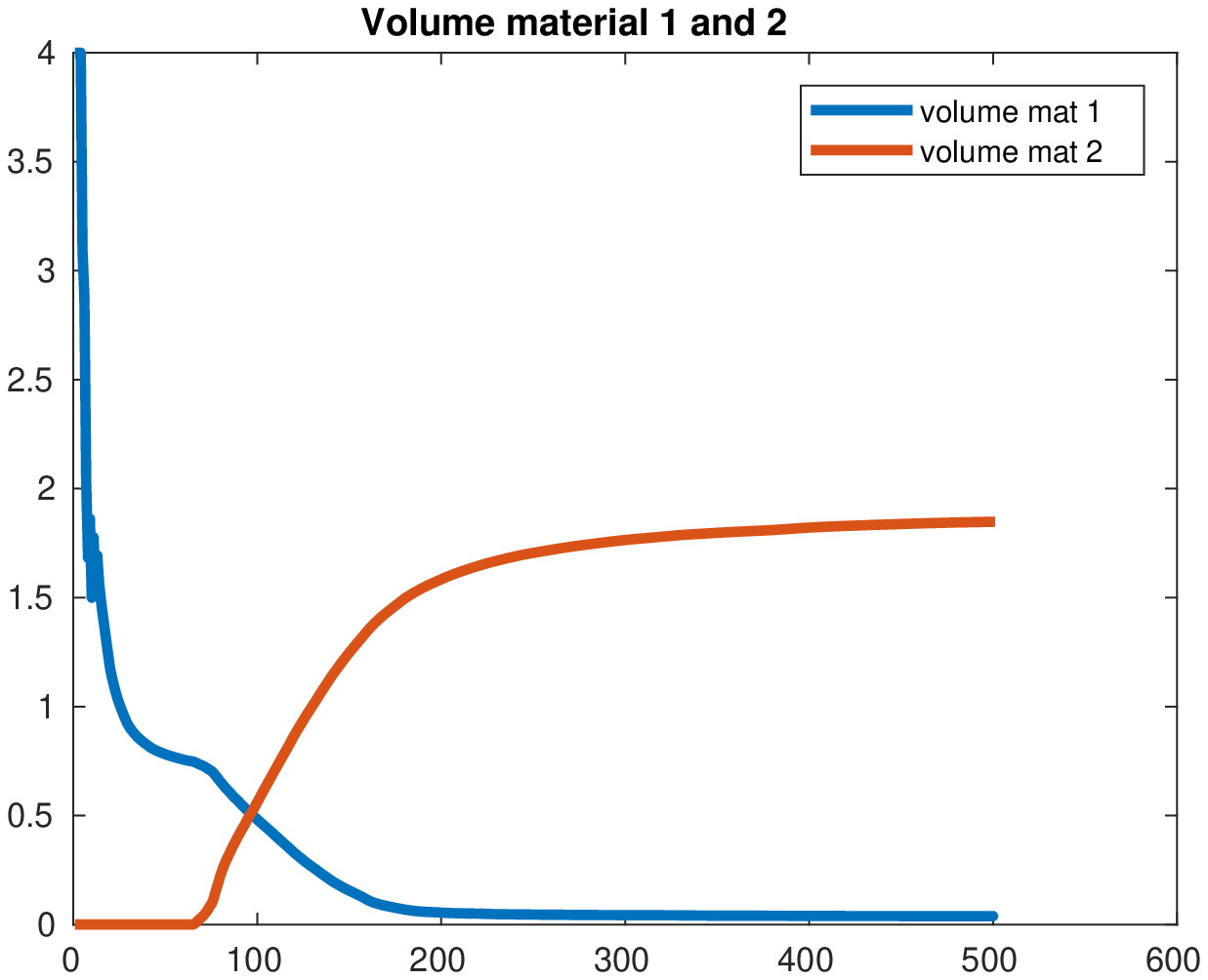}  \\
		(a) & (b)
	\end{tabular}
    \end{center}
	\caption{Evolution of angle and volumes in the course of Algorithm \ref{algo_TDLSMmat} for mast example. (a) Angle $\theta$ between level set function and generalized topological derivative. (b) Volumes of strong material (material 1) and weak material (material 2).}
	\label{fig_graphs_mast}
\end{figure}

\vspace{2mm}

The algorithm introduced in this paper is general and can handle an arbitrary number $M$ of different materials. As mentioned in Remark \ref{rem_interface}, when dealing with PDE-constrained topology optimization problems, the treatment of the interface is important. Also for higher values of $M$, a linear interpolation of the material parameters in elements that are cut by an interface based on the volume fractions similar to \eqref{eq_matInterpol_3mat} is definitely possible, but can be cumbersome as it involves many case distinctions. For this reason, we will illustrate the case of many (here: $M=8$) materials using an academic example without a constraining PDE. We will see that, in this case, we obtain the expected results even when disregarding the material interfaces.

\subsection{Academic Example with M=8} 
We consider an academic multi-material topology optimization problem which is not constrained by a partial differential equation. We search for the optimal distribution of a given number $M$ of materials inside the fixed domain $D = (0,1)^2$. For $i \in \{1, \dots, M\}$ let $f_i: D \rightarrow \mathbb R$ be continuous functions. We aim to minimize the multi-material shape function $\mathcal J: \mathcal A_M \rightarrow \mathbb R$ defined by
\begin{align} \label{eq_topOptiProblem_noPDE}
	\mathcal J(\Omega_1, \dots, \Omega_M) := \sum_{\ell=1}^M \int_{\Omega_\ell} f_\ell(x) \, \mbox dx.
\end{align}
The exact solution for this problem is given by 
\begin{equation*}
	\overline{\Omega_\ell^*} = \{x \in D: f_\ell(x) = \underset{j=1,\dots,M}{\mbox{min}} f_j(x) \}
\end{equation*}
for $\ell =1,\dots,M$. We conducted our experiments for functions $f_\ell$ of the form
\begin{equation*}
	f_\ell(x) = w_\ell \|x-x_m^{(\ell)}\| - d_\ell
\end{equation*}
where $\| \cdot \|$ denotes the Euclidean norm in $\mathbb R^2$, $w_\ell, d_\ell \in \mathbb R$ and $x_m^{(\ell)} \in D$. We chose $M=8$ materials and the data for $w_\ell, d_\ell$ and $x_m^{(\ell)}$ as given in Table \ref{tab_data_example1}. Recall the notation $B(x, \delta)$ for the ball around point $x$ of radius $\delta$. The optimal subdomains are then given by
\begin{align*}
    \Omega_1^* &= (0,1)^2 \setminus \overline{B( (1/2, 1/2)^\top, 0.45)}, \\
    \Omega_3^* &= B( (1/2, 1/2)^\top, 0.2) \setminus \overline{B( (1/2, 1/2)^\top, 0.075)}, \\
    \Omega_4^* &= B( (1/2, 1/2)^\top, 0.075),\\
    \overline{\Omega_i^*} &= \{ x \in D: f_i(x) = \underset{j=1,\dots, 8}{\mbox{min}} f_j(x) \} = \{  x \in D: f_i(x) \leq f_2(x) \}, \, i=5,6,7,8, \\
    \Omega_2^* &= B( (1/2, 1/2)^\top, 0.45) \setminus \overline{B( (1/2, 1/2)^\top, 0.2) \cup \Omega_5^*\cup \Omega_6^*\cup \Omega_7^*\cup \Omega_8^*}.
\end{align*}

We start with a homogeneous material distribution $\Omega_1 = \dots= \Omega_7 = \emptyset$, $\Omega_8 = D$, see Figure \ref{fig_example1_init} (a).
It can be easily verified that the topological derivative of the objective function $\mathcal J$ introduced in \eqref{eq_topOptiProblem_noPDE} for switching material around a point $z \in \Omega_i$ to material $j$ reads
\begin{align*}
    \TD^{i \rightarrow j}(z) = f_j(z) - f_i(z)
\end{align*}
when choosing $\ell(\varepsilon) = |\omega_\varepsilon|$. In this case, the topological derivative is independent of the particular choice of the shape of the inclusion $\omega$.

We applied Algorithm \ref{algo_TDLSMmat} with a constant step size $\kappa =0.5$ and no line search; see Figure \ref{fig_example1_init} (b) for an intermediate design after three iterations. We observe that, for this simple academic example, the angle $\theta$ between the vector-valued level set function $\psi$ and the generalized topological derivative $ G$ decreases monotonically to zero, see Figure \ref{fig_example1_final} (a), and we reach the optimal material distribution, see Figure \ref{fig_example1_final} (b). 

This example serves as a proof of concept and shows that an arbitrary number of materials can be treated. For this example, the treatment of the interfaces between different materials was not critical and we assigned each element of the underlying triangular grid completely to one of the materials based on the value of the vector-valued level set function at the centroid of the triangle. The implementation of this example was done in \texttt{MATLAB}.

\begin{figure}
	\begin{tabular}{cc}
		\includegraphics[width=.5\textwidth]{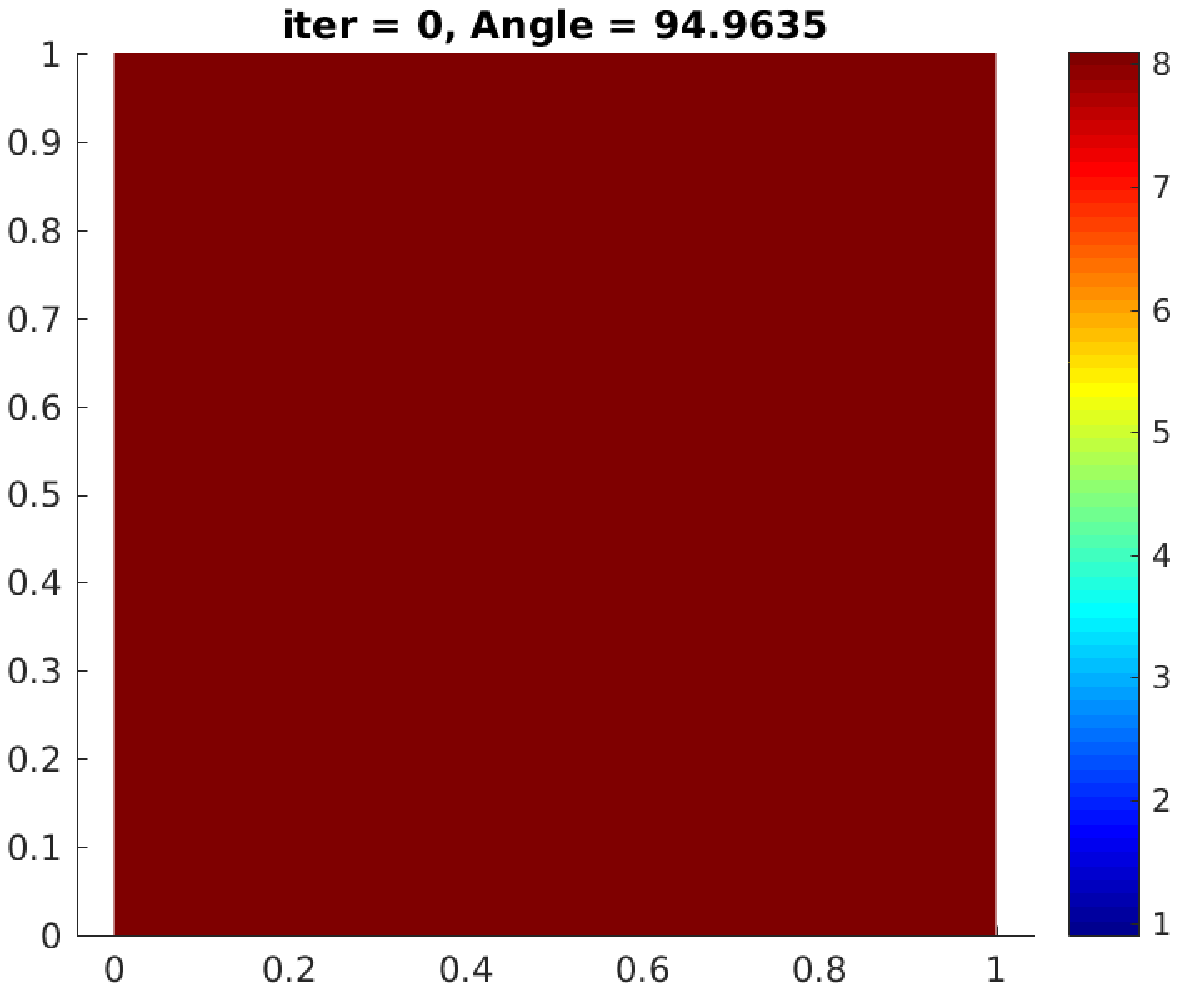}& \includegraphics[width=.5\textwidth]{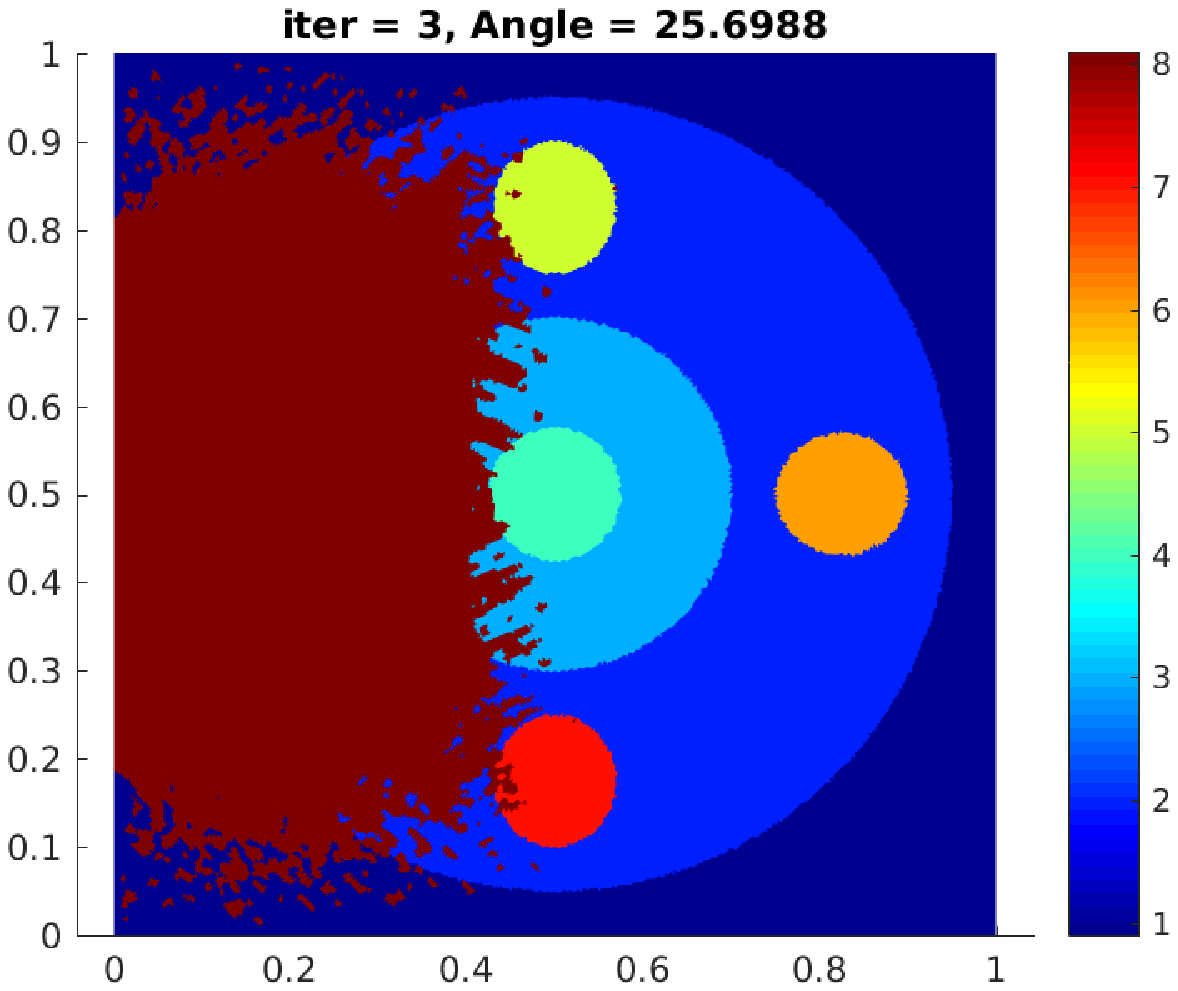} \\
		(a)& (b)
	\end{tabular}
	\caption{(a) Initial design for academic topology optimization problem \eqref{eq_topOptiProblem_noPDE}. (b) Intermediate design after three iterations of Algorithm \ref{algo_TDLSMmat}.}
	\label{fig_example1_init}
\end{figure}

\begin{figure}
	\begin{tabular}{cc}
		 \includegraphics[width=.5\textwidth]{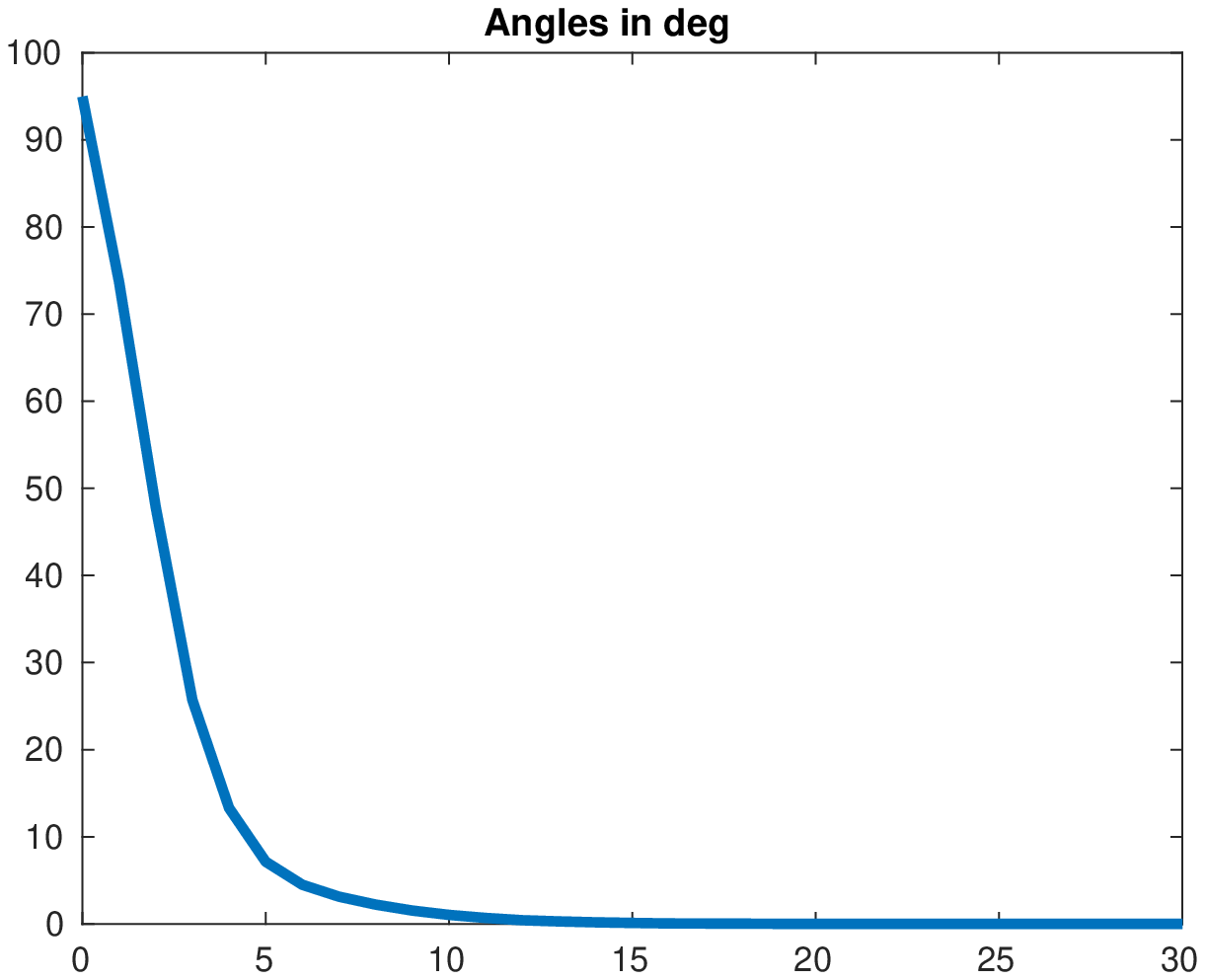} & \includegraphics[width=.5\textwidth]{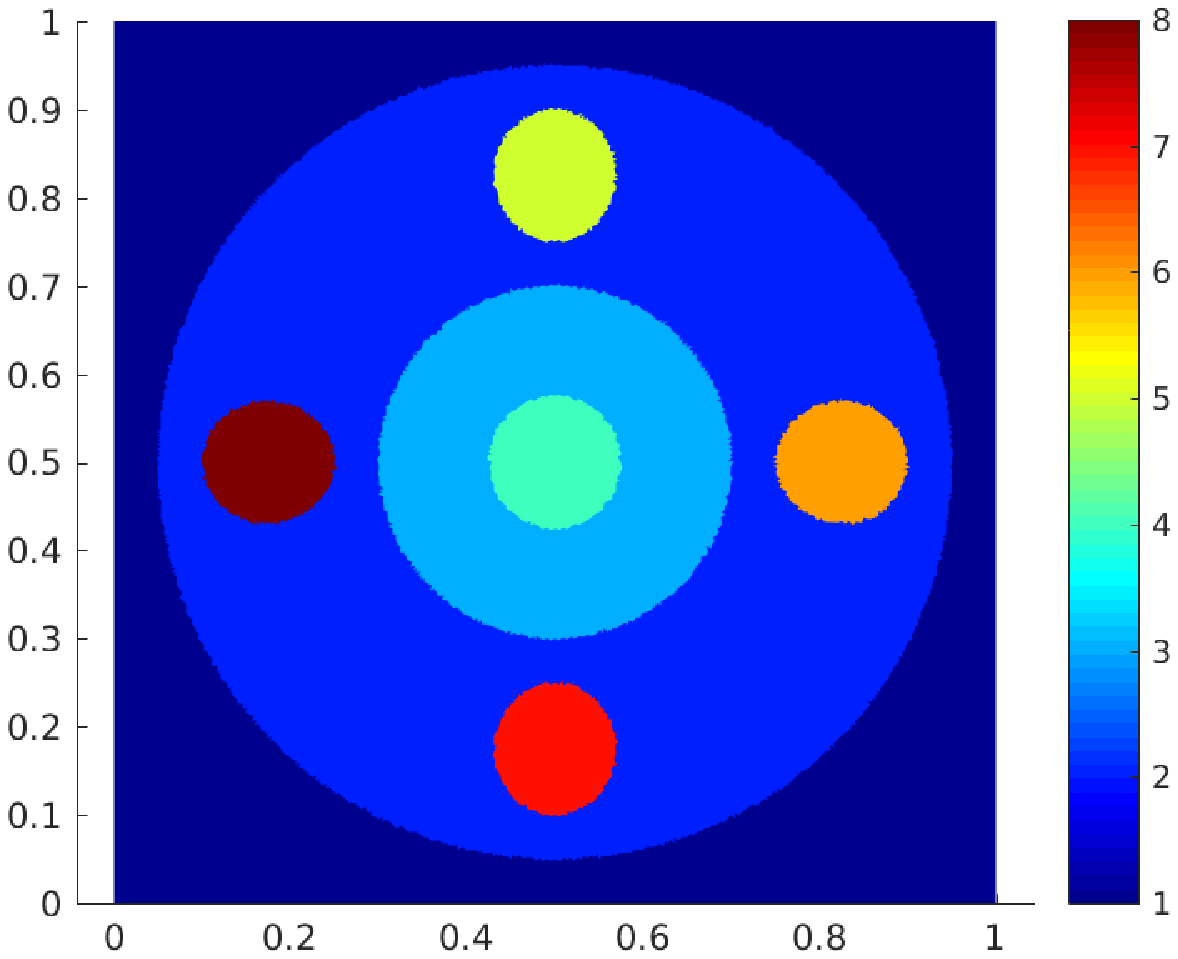} \\
		 (a) & (b)
	\end{tabular}
	\caption{(a) Evolution of angle between vector-valued level set function $\psi$ and generalized topological derivative $ G_\psi$ in optimization. (b) Final design for academic topology optimization problem \eqref{eq_topOptiProblem_noPDE}; angle = $5.9764 \cdot 10^{-6}$ degrees. }
	\label{fig_example1_final}
\end{figure}

\begin{table}
	\begin{center}
	 	\begin{tabular}{c|r|c|l}
			
			$\ell$ & $w_\ell$ & $x_m^{(\ell)}$ & $d_\ell$ \\ \hline \hline
			1 &	0 & -- & 0\\
			2 & 1 & $(0.5, 0.5)^\top$	& 0.45	\\
			3 & $\frac{5}{4}$ & $(0.5, 0.5)^\top$	& 0.5	\\
			4 & $\frac{95}{12}$ & $(0.5, 0.5)^\top$	& 1\\
			5 & 2 & $(0.5, 0.7875)^\top$	& 0.275	\\
			6 & 2 & $(0.7875, 0.5)^\top$	& 0.275	\\
			7 & 2 & $(0.5, 0.2125)^\top$	& 0.275	\\
			8 & 2 & $(0.2125, 0.5)^\top$	& 0.275	\\ 
		\end{tabular}
	\end{center}
	\caption{Data used for problem \eqref{eq_topOptiProblem_noPDE}.}
	\label{tab_data_example1}
\end{table}

\section*{Conclusion and Outlook}
We introduced a multi-material topology optimization method using a level-set framework, which is based on the concepts of topological derivatives. The idea was to represent a design consisting of $M$ different phases by means of a vector-valued level set function mapping into $\mathbb R^{M-1}$. Dividing the Euclidean space $\mathbb R^{M-1}$ into $M$ sectors, each sector corresponds to one material. Based on this splitting we introduced a local optimality condition for a design consisting of $M$ materials using topological derivatives. The proposed algorithm strives to reach this optimality condition using a fixed point iteration. We illustrated the potential of the algorithm for classical compliance minimization problems with $M=3$ materials as well as an academic example consisting of $M=8$ materials. We believe that, besides applications of topology optimization, the method could also be useful in applications from mathematical image segmentation, e.g. for treating a Mumford-Shah type functional as it was done in \cite{HintermuellerLaurain2009}.

\section*{Acknowledgement}
The author has been supported by Graz University of Technology and would like to gratefully acknowledge the funding. Moreover, the author is grateful for the discussions on the topic of the paper with Samuel Amstutz and to Charles Dapogny for making his FreeFem++ code of various topology optimization approaches publicly available at \cite{DapognyCode2018}.

\appendix
\section[A]{MATLAB implementation for data structure} \label{appendix_code}
\definecolor{mygreen}{rgb}{0,0.6,0}
\definecolor{mygray}{rgb}{0.5,0.5,0.5}
\definecolor{mymauve}{rgb}{0.58,0,0.82}

\lstset{ 
  backgroundcolor=\color{white},   % choose the background color; you must add \usepackage{color} or \usepackage{xcolor}; should come as last argument
  basicstyle=\footnotesize,        % the size of the fonts that are used for the code
  breakatwhitespace=false,         % sets if automatic breaks should only happen at whitespace
  breaklines=true,                 % sets automatic line breaking
  captionpos=b,                    % sets the caption-position to bottom
  commentstyle=\color{mygreen},    % comment style
  deletekeywords={...},            % if you want to delete keywords from the given language
  escapeinside={\%*}{*)},          % if you want to add LaTeX within your code
  extendedchars=true,              % lets you use non-ASCII characters; for 8-bits encodings only, does not work with UTF-8
  firstnumber=1000,                % start line enumeration with line 1000
  frame=single,	                   % adds a frame around the code
  keepspaces=true,                 % keeps spaces in text, useful for keeping indentation of code (possibly needs columns=flexible)
  keywordstyle=\color{blue},       % keyword style
  language=Octave,                 % the language of the code
  morekeywords={*,...},            % if you want to add more keywords to the set
  numbers=left,                    % where to put the line-numbers; possible values are (none, left, right)
  numbersep=5pt,                   % how far the line-numbers are from the code
  numberstyle=\tiny\color{mygray}, % the style that is used for the line-numbers
  rulecolor=\color{black},         % if not set, the frame-color may be changed on line-breaks within not-black text (e.g. comments (green here))
  showspaces=false,                % show spaces everywhere adding particular underscores; it overrides 'showstringspaces'
  showstringspaces=false,          % underline spaces within strings only
  showtabs=false,                  % show tabs within strings adding particular underscores
  stepnumber=2,                    % the step between two line-numbers. If it's 1, each line will be numbered
  stringstyle=\color{mymauve},     % string literal style
  tabsize=2,	                   % sets default tabsize to 2 spaces
  title=\lstname                   % show the filename of files included with \lstinputlisting; also try caption instead of title
}

\begin{lstlisting}[language=Octave, title=createDataStructure.m]
function [normalIdces, normals] = createDataStructure(M)
%Creates data structure necessary for extracting normal
%vectors of hyperplanes defining M sectors of space R^{M-1}
% Input:
%   M ... integer, M>1, number of different materials
% Output:
%   normalIdces ... M x (M-1) (absolute values between 1 and 
%           nMat*(nMat-1)/2; sign of normalIdces(l,k) indicates if 
%           corresponding normal is pointing into sector l (+) or 
%           out of sector l (-))
%   normals ... M*(M-1)/2 x (M-1); each row corresponds 
%           to normal vector of a hyperplane

    eps = 1e-10;    %numerical tolerance
    %%%%%%%%%%%%%%%%%%%%%%%%%%%%%%%%%%%%%%%%%%%%
    %Step 1: set up M points p1,...,pM in R^{M-1} (p0 = (0,...,0) not stored)
    p=zeros(M,M-1);
    p(1:M-1,1:M-1) = eye(M-1);
    p(M,:) = -1*ones(1,M-1);

    %%%%%%%%%%%%%%%%%%%%%%%%%%%%%%%%%%%%%%%%%%%%
    %Step 2: Define M*(M-1)/2 hyperplanes
    
    %Step 2a: For each of the M sectors save M-1 points
    pIdxSec = zeros(M,M-1);     
    for l=1:M
        count = 0;
        for i=1:M
            if l ~= i
                count = count+1;
                pIdxSec(l,count)=i;
            end
        end
    end
    
    %Step 2b: Set up hyperplanes; each hyperplane is given by any (M-2) points out of p1,...,pM (plus origin p0)
    nHyPl = M*(M-1)/2;
    nothyperplanes = zeros(nHyPl,2);    
    count = 1;
    for i=1:M
        for j=i+1:M
            nothyperplanes(count,1)=i;
            nothyperplanes(count,2)=j;
            count = count+1;
        end
    end
    hyperplanes = zeros(nHyPl,M-2);
    for k=1:nHyPl
        count = 1;
        for j=1:M
            if ~ismember(j, nothyperplanes(k,:))
                hyperplanes(k, count) = j;
                count = count+1;
            end
        end
    end
    
    %%%%%%%%%%%%%%%%%%%%%%%%%%%%%%%%%%%%%%%%%%%%
    %Step 3: For each sector, collect indices of bounding hyperplanes
    normalIdces = -(nHyPl+1)*ones(M,M-1); %initialize with a value that will not be attained
    for k=1:nHyPl
        for l=1:M
            if prod(ismember(hyperplanes(k,:),pIdxSec(l,:))) %if hyperplane k is part of closure of sector l
                j=1;
                while normalIdces(l,j)>=0
                    j = j+1;
                end
                normalIdces(l,j) = k;
            end
        end
    end
    assert( min(min(normalIdces)) > 0);
    
    %%%%%%%%%%%%%%%%%%%%%%%%%%%%%%%%%%%%%%%%%%%%
    %Step 4: Compute normal vectors of hyperplanes
    normals = zeros(nHyPl,M-1);
    for k=1:nHyPl
       mat = zeros(M-2, M-1);
       for j=1:M-2
           mat(j,:) = p(hyperplanes(k,j),:);
       end
       normals(k,:) = null(mat)'; %normal vector is null space of matrix mat
    end
        
    %add signs to normalIdces such that sign(..)normals points into sector
    for l=1:M
        for j=1:M-1
            idxNormal = normalIdces(l,j);
            nVec = normals(idxNormal,:);
            for i=1:M-1
                if nVec*p(pIdxSec(l,i),:)' < -eps
                    normalIdces(l,j) = -1*normalIdces(l,j);
                    break;
                end
            end
        end
    end
    %%%%%%%%%%%%%%%%%%%%%%%%%%%%%%%%%%%%%%%%%%%%

end

\end{lstlisting}
\begin{lstlisting}[language=Octave, title=getNormal.m]
function [ nij ] = getNormal( i, j, normalIdces, normals )
%Returns normal vector pointing out of sector i and into sector j
%   function uses data structure obtained by 
%   'createDataSturcture(...)' to return the correct normal 
%   vector between sectors i and j
% Input: 
%   i ... index of sector (between 1 and number of materials nMat)
%   j ... index of sector (between 1 and number of materials nMat)
%   normalIdces ... nMat x (nMat-1) (absolute values between 1 and 
%           nMat*(nMat-1)/2; sign of normalIdces(l,k) indicates if 
%           corresponding normal is pointing into sector l (+) or 
%           out of sector l (-))
%   normals ... nMat*(nMat-1)/2 x (nMat-1); each row corresponds 
%           to normal vector of a hyperplane
% Output:
%   nij ... normal vector pointing from sector j into sector i
    assert(i ~= j && 0 < i && i <= size(normalIdces,1) && 0 < j && j <= size(normalIdces,1) );
    for l=1:size(normalIdces,2)
        if ismember(abs(normalIdces(j, l)), abs(normalIdces(i,:)))
            k = l;
            break;
        end
    end
    nij = sign(normalIdces(j, k))*normals( abs(normalIdces(j, k)),:);
end

\end{lstlisting}
\begin{lstlisting}[language=Octave, title = isInSector.m]
function [banswer] = isInSector(l, p, normals, normalIdces)
% Determines if point p is in sector l given data structure 
% obtained by 'createDataSturcture(...)'
% Input: 
%   l ... index of sector (between 1 and number of materials nMat)
%   p ... point in R^(nMat-1); must be column vector
%   normalIdces ... nMat x (nMat-1) (absolute values between 1 and 
%           nMat*(nMat-1)/2; sign of normalIdces(l,k) indicates if 
%           corresponding normal is pointing into sector l (+) or 
%           out of sector l (-))
%   normals ... nMat*(nMat-1)/2 x (nMat-1); each row corresponds 
%           to normal vector of a hyperplane
    assert(0 < l && l <= size(normalIdces,1) );
    nMat = size(normalIdces,1);
    for j=1:nMat-1  %nMat-1 is the number of hyperplanes bounding a sector
        idx = normalIdces(l,j); %idx has a sign
        if sign(idx)*normals(abs(idx),:)*p < -1e-10
            banswer = false;
            return;
        end       
    end
    banswer = true;
end

\end{lstlisting}

\bibliography{multimat}
\bibliographystyle{plain}

\end{document}